\documentclass[12pt,reqno]{amsart}
\usepackage{amsfonts,amsrefs,latexsym,amsmath, amssymb, mathrsfs, verbatim}
\usepackage{url,color}
\usepackage{pifont}
\usepackage{upgreek}
\usepackage{fancyhdr}
\usepackage{hyperref}
\usepackage{calligra}
\usepackage{marvosym}
\usepackage[percent]{overpic}
\usepackage{pict2e}

\usepackage[hypcap=false]{caption}
\usepackage{xcolor}
\usepackage{wasysym}
\usepackage{accents}

\topmargin - .23 in

\newtheorem{theorem}{Theorem}[section]
\newtheorem{proposition}[theorem]{Proposition}
\newtheorem{lemma}[theorem]{Lemma}

\theoremstyle{definition}
\newtheorem{definition}[theorem]{Definition}
\newtheorem{remark}[theorem]{Remark}

\theoremstyle{plain}

\DeclareMathAlphabet{\mathcalligra}{T1}{calligra}{m}{n}
\DeclareFontShape{T1}{calligra}{m}{n}{<->s*[2.2]callig15}{}

\newcommand{\leftexp}[2]{{\vphantom{#2}}^{#1}{#2}}

\numberwithin{equation}{subsection}

\begin{document}
\title{Finite-time degeneration of hyperbolicity without blowup for quasilinear wave equations}
\author[JS]{Jared Speck$^{* \dagger}$}

\thanks{$^{\dagger}$Speck gratefully acknowledges support from NSF grant \# DMS-1162211,
from NSF CAREER grant \# DMS-1454419,
from a Sloan Research Fellowship provided by the Alfred P. Sloan foundation,
and from a Solomon Buchsbaum grant administered by the Massachusetts Institute of Technology.
}

\thanks{$^{*}$Massachusetts Institute of Technology, Department of Mathematics, 77 Massachusetts Ave, Room 2-265, Cambridge, MA 02139-4307, USA. \texttt{jspeck@math.mit.edu}}

\begin{abstract}
In three spatial dimension, we study the Cauchy problem 
for the wave equation
$- \partial_t^2 \Psi + (1 + \Psi)^P \Delta \Psi = 0$
for $P \in \lbrace 1,2 \rbrace$.
We exhibit a form of stable Tricomi-type degeneracy formation  
that has not previously been studied 
in more than one spatial dimension.
Specifically, using only energy methods and ODE techniques,
we exhibit an open set of data such that $\Psi$ is initially near $0$
while $1 + \Psi$ vanishes in finite time.
In fact, generic data, 
when appropriately rescaled, 
lead to this phenomenon.
The solution remains regular in the following sense: 
there is a high-order $L^2$-type energy, 
featuring degenerate weights only at the top-order,
that remains bounded.
When $P=1$, we show that any $C^1$ extension of $\Psi$
to the future of a point where $1 + \Psi = 0$ 
must exit the regime of hyperbolicity.
Moreover, the Kretschmann scalar 
of the Lorentzian metric corresponding to the wave equation
blows up at those points.
Thus, our results show that curvature blowup does not always coincide 
with singularity formation in the solution variable.
Similar phenomena occur when $P=2$, where
the vanishing of $1 + \Psi$ corresponds 
to the failure of strict hyperbolicity, 
although the equation is hyperbolic at all values of $\Psi$.

The data are compactly supported and
are allowed to be large or small as measured by an
unweighted Sobolev norm. However, we assume that initially, 
the spatial derivatives of $\Psi$ are nonlinearly small
relative to $|\partial_t \Psi|$, 
which allows us to treat the equation
as a perturbation of the ODE 
$
\displaystyle
\frac{d^2}{dt^2}
\Psi = 0
$.
We show that for appropriate data,
$\partial_t \Psi$ remains quantitatively negative,
which simultaneously drives the degeneracy formation
and yields a favorable spacetime integral in the energy estimates
that is crucial for controlling some top-order error terms.
Our result complements those of Alinhac and Lindblad, who 
showed that if the data are small as measured by a Sobolev norm with radial weights,
then the solution is global.

\bigskip

\noindent \textbf{Keywords}: degenerate hyperbolic, strictly hyperbolic, Tricomi equation, weakly hyperbolic

\bigskip

\noindent \textbf{Mathematics Subject Classification (2010)} Primary: 35L80; Secondary: 35L05, 35L72
\end{abstract}

\maketitle

\centerline{\today}


\section{Introduction}
\label{S:INTRO}
Many authors have studied model nonlinear wave equations 
as a means to gain insight into more challenging 
wave-like quasilinear equations, 
such as Einstein's equations of general relativity, 
the compressible Euler equations without vorticity, 
and the equations of elasticity.
Motivated by the same considerations, 
in this paper, we study model quasilinear wave equations
in three spatial dimensions.
To simplify the presentation,
we have chosen to restrict our attention to the $3$-dimensional case only;
with only modest additional effort,
our results could be generalized to apply in any number of spatial dimensions.
In a broad sense, we are interested in finding initial conditions
without symmetry assumptions that lead to some kind of stable breakdown.
In our main results, which we summarize just below \eqref{E:SPACETIMEMETRIC}, 
we exhibit a type of stable degenerate solution behavior, \emph{distinct from blowup},
that, to the best of our knowledge, has not previously been studied 
in the context of quasilinear equations in more than one spatial dimension.
Roughly, we show that there exists an open set of data
such that certain principal coefficients in the equation
vanish in finite time \emph{without a singularity forming in the solution}.
More precisely, the vanishing of the coefficients corresponds to the vanishing of the wave speed,
which in turn is tied to other kinds of degeneracies described below.
We note that Wong \cite{wW2016} has obtained similar constructive results 
for axially symmetric timelike minimal submanifolds of Minkowski spacetime,
a setting in which the equations of motion are a system of 
(effectively one-space-dimensional)
quasilinear wave equations with principal coefficients that depend on the solution
(but not its derivatives). Specifically, he showed that \emph{all} axially symmetric solutions
(without any smallness assumption) lead to a finite-time degeneracy caused by
the vanishing of a principal coefficient in
the evolution equations. We also note that in the case of one spatial dimension, 
results similar to ours are obtained in \cites{kKyS2013,yS2013,sY2016a,sY2016b} 
using proofs by contradiction that rely on the method of Riemann invariants.
However, since the method of Riemann invariants is not applicable
in more than one spatial dimension and since we are interested in direct proofs,  
our approach here is quite different.

Through our study of model problems, we are aiming to develop approaches
that might be useful for studying the kinds of degeneracies that might develop
in solutions to more physically relevant quasilinear equations.
One consideration behind this aim is that there are relatively few 
breakdown results for quasilinear equations 
compared to the semilinear case. A second consideration is that
many of the techniques that have been used to study semilinear
wave equations do not apply in the quasilinear case;
see Subsect.\ \ref{SS:FURTHERDISCUSSIONANDMOTIVATION} for further discussion.
A third consideration concerns fundamental limitations of semilinear model equations:
they are simply incapable of exhibiting some of the most important degeneracies that
can occur in solutions to quasilinear equations. 
In particular, the degeneracy exhibited by the solutions from our main results
cannot occur in solutions to semilinear wave equations with principal part
equal to the linear wave operator\footnote{Here,
$\square_m := -\partial_t^2 + \Delta$ denotes the standard linear wave
operator corresponding to the Minkowski metric $m := \mbox{\upshape diag}(-1,1,1,1)$ on $\mathbb{R}^{1+3}$.} 
$\square_m$. As a second example of breakdown that is unique to the quasilinear case, 
we note that the phenomenon of shock formation, 
described in more detail at the end of Subsect.\ \ref{SS:OTHERDEGNERATEHYPERBOLICEQUATIONS},
cannot occur in solutions to semilinear equations since, in the semilinear case,
the evolution of characteristics is not influenced in any way\footnote{In the works on shock formation 
for quasilinear equations described in Subsect.\ \ref{SS:OTHERDEGNERATEHYPERBOLICEQUATIONS},
the intersection of the characteristics is tied to the blowup of some derivative of the solution.} 
by the solution. 

In view of the above discussion, it is significant that 
our analysis has robust features and could be extended to apply to
a large class\footnote{Of course, we can only hope to treat
wave equations whose principal spatial coefficients vanish 
when evaluated at some finite values of the solution variable;
equations such as $- \partial_t^2 \Psi + (1 + \Psi^2) \Delta \Psi = 0$
are manifestly immune to the kind of degeneracies under study here.} 
of quasilinear equations.
The robustness stems from the fact that our proofs are based 
only on energy estimates, ODE-type estimates, 
and the availability of an important monotonic spacetime integral (which we describe below)
that arises in the energy estimates.
However, rather than formulating a theorem
about a general class of equations, we prefer to keep the paper short
and to exhibit the main ideas by studying
only the model equation \eqref{E:MODELWAVE} below 
in the cases $P=1,2$.

\subsection{Statement of the equations and summary of the main results}
\label{SS:STATEMENTOFEQUATIONS}
Specifically, in the cases $P=1,2$,
we study the following model Cauchy problem on $\mathbb{R}^{1+3}$:
\begin{subequations}
\begin{align} \label{E:MODELWAVE}
	- \partial_t^2 \Psi
	+ (1 + \Psi)^P \Delta \Psi
	& = 0,
		\\
	(\Psi|_{\Sigma_0},\partial_t \Psi|_{\Sigma_0})
	& = (\mathring{\Psi},\mathring{\Psi}_0),
	\label{E:MODELDATA}
\end{align}
\end{subequations}
where $(x^0:=t,x^1,x^2,x^3)$ is a fixed set of standard rectangular coordinates on $\mathbb{R}^{1+3}$,
$\Delta := \sum_{a=1}^3 \partial_a^2$ is the standard Euclidean Laplacian on $\mathbb{R}^3$, 
and throughout,\footnote{Here we use the notation ``$\simeq$'' to mean ``diffeomorphic to.''} 
$\Sigma_t := \lbrace t \rbrace \times \mathbb{R}^3 \simeq \mathbb{R}^3$.
We sometimes denote the spatial coordinates by $\underline{x} :=(x^1,x^2,x^3)$.
Note that we can rewrite \eqref{E:MODELWAVE} as\footnote{Throughout we use Einstein's summation convention.} 
$
(g^{-1})^{\alpha \beta}(\Psi)\partial_{\alpha} \partial_{\beta} \Psi = 0
$,
where $g$ is the Lorentzian (for $\Psi > - 1$) metric
\begin{align} \label{E:SPACETIMEMETRIC}
	g
	:= - dt^2 + (1 + \Psi)^{-P} \sum_{a=1}^3 (dx^a)^2.
\end{align}
This geometric perspective will be useful at various points in our discussion.

We now summarize our results; see Theorem~\ref{T:STABLEFINITETIMEBREAKDOWN}
and Prop.\ \ref{P:BLOWUPOFKRETSCHMANN} for precise statements.

\begin{changemargin}{.25in}{.25in}
	\textbf{Summary of the main results.}
	In the case $P=1$,
	there exists an open subset 
	of $H^6(\mathbb{R}^3) \times H^5(\mathbb{R}^3)$
	comprising compactly supported initial data 
	$(\mathring{\Psi},\mathring{\Psi}_0)$
	such that the solution $\Psi$, its \emph{spatial} derivatives,
	and its \emph{mixed space-time} derivatives
	initially satisfy a nonlinear smallness
	condition compared to\footnote{Throughout, $[p]- := |\min \lbrace p,0 \rbrace|$.} 
	$\max_{\Sigma_0}[\mathring{\Psi}_0]_-$
	and 
	$
	\displaystyle
	\frac{1}{\| \mathring{\Psi}_0 \|_{L^{\infty}(\Sigma_0)}}
	$, 
	and such that the solution has the following property: 
	the coefficient $1 + \Psi$ 
	in \eqref{E:MODELWAVE} vanishes at some time 
	$T_{\star} \in (0, \infty)$.
	In fact, the finite-time vanishing of $1 + \Psi$ \emph{always} occurs 
	if $\mathring{\Psi}_0$ is non-trivial
	and the data 
	are appropriately rescaled;
	see Remark~\ref{R:GENERICDEGENERACY}.
	Moreover, 
	\[
	\Psi
		\in 
			C\left([0,T_{\star}),H^6(\mathbb{R}^3) \right)
			\cap
			L^2\left([0,T_{\star}],H^6(\mathbb{R}^3) \right)
			\cap
			C\left([0,T_{\star}],H^5(\mathbb{R}^3) \right)
	\]
	while for any $N < 5$,
	\[
	\partial_t \Psi
	\in 
	C\left([0,T_{\star}),H^5(\mathbb{R}^3) \right)
	\cap
	L^{\infty}\left([0,T_{\star}],H^5(\mathbb{R}^3) \right)
	C\left([0,T_{\star}],H^N(\mathbb{R}^3) \right).
	\]
	In addition, the Kretschmann scalar
	$Riem(g)^{\alpha \beta \gamma \delta} Riem(g)_{\alpha \beta \gamma \delta}$
	blows up precisely at points where $1 + \Psi$ vanishes,
	where $Riem(g)$ denotes the Riemann curvature of $g$.
	Finally, the solution exits the regime of hyperbolicity
	at time $T_{\star}$ and thus it cannot be continued beyond $T_{\star}$ as a classical solution to 
	a hyperbolic equation.
	In the case $P=2$,
	similar results hold,
	the main differences being that $\Psi$ is not necessarily an element of
	$L^2\left([0,T_{\star}],H^6(\mathbb{R}^3) \right)$
	and that the strict hyperbolicity\footnote{Equation \eqref{E:MODELWAVE}
	is said to be strictly hyperbolic in the direction $\upomega$ if the symbol 
	$p(\upxi) := -\upxi_0^2 + (1 + \Psi)^P \sum_{a=1}^3 \upxi_a^2$ 
	has the following property: for any one-form $\upxi \neq 0$, 
	the polynomial $s \rightarrow p(\upxi + s \upomega)$
	has two distinct real roots. It is straightforward to see that equation \eqref{E:MODELWAVE}
	is strictly hyperbolic in the direction $\upomega := (1,0,0,0)$ if $1 + \Psi > 0$, 
	and that it is \emph{not} strictly hyperbolic in any direction if $1 + \Psi = 0$.
	\label{FN:STRICTLYHYPERBOLIC}} 
	breaks down when $1 + \Psi$ vanishes
	but hyperbolicity\footnote{Here, by hyperbolic (in the direction $\upomega$), 
	we mean that for all one-forms $\upxi \neq 0$,
	the polynomial $s \rightarrow p(\upxi + s \upomega)$
	from Footnote~\ref{FN:STRICTLYHYPERBOLIC}
	has only real roots. Such polynomials are known as \emph{hyperbolic polynomials}. \label{FN:HYPERBOLIC}} 
	does not.\footnote{In the literature, 
	equations exhibiting this kind of degeneracy
	are often referred to as \emph{weakly hyperbolic}.}  
	This leaves open, in the case $P=2$, the possibility of classically
	extending the solution past time $T_{\star}$;
	see Subsubsect.\ \ref{SSS:EXTENDINGPASTHEDEGENERACY}.
\end{changemargin}

\subsection{Paper outline}
\label{SS:PAPEROUTLINE}
The remainder of the paper is organized as follows.
\begin{itemize}
	\item In Subsect.\ \ref{SS:PRELIMINARYDISCUSSIONOFRESULTS},
		we provide some initial remarks expanding upon various aspects of our
		results.
	\item In Subsect.\ \ref{SS:FURTHERDISCUSSIONANDMOTIVATION},
		we mention some techniques that have been used in studying
		the breakdown of solutions to semilinear equations. 
		As motivation for the present work, 
		we point out some limitations of the semilinear techniques
		for the study of quasilinear equations.
		\item In Subsect.\ \ref{SS:BRIEFOVERVIEW} we provide a brief overview of the proof
		of our main results.
	\item In Subsect.\ \ref{SS:OTHERDEGNERATEHYPERBOLICEQUATIONS},
		we describe some connections between our results and prior work
		on degenerate hyperbolic PDEs.
	\item In Subsect.\ \ref{SS:NOTATION} we summarize our notation.
\item In Sect.\ \ref{S:DATAANDBOOTSTRAP}, we state our assumptions on the initial data
	and introduce bootstrap assumptions.
	\item In Sect.\ \ref{S:APRIORIESTIMATES}, we use the bootstrap and data-size assumptions
		of Sect.\ \ref{S:DATAANDBOOTSTRAP} to derive a priori pointwise estimates and energy estimates.
		From the energy estimates, we deduce improvements of the bootstrap assumptions.
	\item In Sect.\ \ref{S:MAINRESULTS}, we use the 
	 estimates of Sect.\ \ref{S:APRIORIESTIMATES} to prove our main results.
\end{itemize}

\subsection{Initial remarks on the main results}
\label{SS:PRELIMINARYDISCUSSIONOFRESULTS}
As far as we know, there are no prior results
in the spirit of our main results in more than one spatial dimension.
There are, however, examples in which the Cauchy problem 
for a quasilinear wave equation has been solved
(for suitable data without symmetry assumptions)
and such that it was shown that some derivative of the solution
blows up in finite time while the solution itself remains bounded.
One class of such examples comprises
shock formation results, which we describe in more detail at the end of Subsect.\ \ref{SS:OTHERDEGNERATEHYPERBOLICEQUATIONS}.
A second example is Luk's work \cite{jL2013} on the formation of weak null singularities
in a family of solutions to the Einstein-vacuum equations.
Specifically, he exhibited a stable family of solutions
such that the Christoffel symbols (which are, roughly speaking, the first derivatives of the solution) 
blow up along a null boundary while the metric (that is, the solution itself) 
extends continuously past the null boundary. 
We stress that the degeneracy we have exhibited in our main results 
is much less severe than in the above results;
there is no blowup in our solutions, except possibly
at the top derivative level,
due to the degeneracy of the weights in the energy~\eqref{E:INTROENERGYDEF}.

We also point out a connection between our work here and
our joint works \cites{iRjS2014a,iRjS2014b} with Rodnianski, 
in which we proved stable blowup results (without symmetry assumptions) for 
solutions to the linearized and nonlinear Einstein-scalar field and Einstein-stiff fluid systems.
In the nonlinear problem, the wave speed became, 
relative to a geometrically defined coordinate system,\footnote{Specifically, the $\Sigma_t$
have constant mean curvature and the spatial coordinates are transported along the unit normal
to $\Sigma_t$.}  
\emph{infinite} at the singularity.
Although the infinite
wave speed is in the opposite direction of the degeneracy exhibited by our
main results (in which the wave speed vanishes\footnote{Note that the effective 
wave speed for equation \eqref{E:MODELWAVE} is $(1 + \Psi)^{P/2}$.}), 
the analysis in \cites{iRjS2014a,iRjS2014b} shares a key feature with that of the present work:
the solution regime studied is such that the time derivatives dominate 
the evolution. That is, the spatial derivatives remain negligible, all the way up 
to the degeneracy; see Subsect.\ \ref{SS:BRIEFOVERVIEW} for further discussion
regarding this issue for the solutions under study here. 
Hence, those works and the present
work all exhibit the stability of ODE-type behavior in some solutions
to wave equations.

\subsubsection{Remarks on small data}
\label{SSS:SMALLDATAGLOBALEXISTENCE}
The methods of Alinhac \cite{sA2003} and Lindblad \cite{hL2008}
yield that small-data solutions to equation \eqref{E:MODELWAVE}
exist globally,\footnote{Since the equations do not satisfy the null condition,
the asymptotics of the solution can be distorted compared to the case of solutions
to the linear wave equation.} 
where the size of the data is measured by
Sobolev norms with \emph{radial weights}.
Consequently, if
$(\mathring{\Psi},\mathring{\Psi}_0)$
are compactly supported data 
to which our main results apply, then
for $\uplambda$ sufficiently large,
the solution corresponding to the data
$(\uplambda^{-1} \mathring{\Psi},\uplambda^{-1} \mathring{\Psi}_0)$
is global. On the other hand, our main results apply 
to data that are allowed to be small in certain unweighted norms,
as long as the spatial derivatives are ``very small.''
How can we reconcile these two competing statements?
The answer is that our data assumptions are nonlinear in nature and
are \emph{not} invariant under the rescaling
$
(\mathring{\Psi},\mathring{\Psi}_0)
\rightarrow
(\uplambda^{-1} \mathring{\Psi},\uplambda^{-1} \mathring{\Psi}_0)
$
if $\uplambda$ is too large.
We can sketch the situation as follows 
(see Subsect.\ \ref{SS:SMALLNESSASSUMPTIONS} for the precise nonlinear smallness assumptions that we use to close our proof):
if $\epsilon$ is the size of $\nabla \Psi$ at time $0$
(where $\nabla$ denotes the spatial coordinate gradient)
and $\delta$ is the size of $\partial_t \Psi$ at time $0$,
then, roughly speaking, some parts 
of our proof rely on\footnote{For example, 
a careful analysis of the proof of inequality \eqref{E:MAINAPRIORIENERGYESTIMATE}
yields that the constant ``$C$'' in front of the $\mathring{\upepsilon}^2$
term on the right-hand side depends on 
$\exp \left(\mathring{\updelta}_*^{-1} \right)$,
where $\mathring{\updelta}_*$
is defined in \eqref{E:CRUCIALDATASIZEPARAMETER}.
See Subsect.\ ~\ref{SS:CONVENTIONSFORCONSTANTS}
for our conventions regarding the dependence of constants on various parameters.} 
the assumption that
$\epsilon \exp \left(C \delta^{-1} \right) \lesssim 1$. 
The point is that if $\uplambda$ is too large, 
then the assumption is not satisfied, the reason being that
$\epsilon$ and $\delta$ both scale like $\uplambda^{-1}$.
One can contrast this against the discussion in Subsect.\ \ref{SS:EXISTENCEOFDATA},
where we note that a different scaling of the data
always leads to our nonlinear smallness assumptions being satisfied.

\subsubsection{Remarks on extending the solution past the degeneracy}
\label{SSS:EXTENDINGPASTHEDEGENERACY}
	It is of interest to know if and when the solutions 
	provided by our main results can be extended, 
	as solutions with some kind of Sobolev regularity,\footnote{The Cauchy--Kovalevskaya theorem 
	could be used to prove an
	(admittedly unsatisfying) result
	showing that in the cases $P=1,2$, one can extend \emph{analytic} solutions to
	equation \eqref{E:MODELWAVE} to exist in a spacetime neighborhood
	of a point at which $1 + \Psi$ vanishes.
	Note that this shows that the blowup of the curvature of the metric of equation \eqref{E:MODELWAVE}
	that occurs when $1 + \Psi = 0$ 
	is not always an obstacle to continuing the solution
	classically.} 
	past the time of first vanishing of $1 + \Psi$.
	Although we do not address this question in this article,
	in this subsubsection, we describe what is known and some of the difficulties
	that one would encounter in attempting to answer it.
	The cases $P=1$ and $P=2$ in equation \eqref{E:MODELWAVE} 
	correspond to different phenomena and hence we will
	discuss them separately, starting with the case $P=1$.
	
	Interesting results have recently been obtained in \cite{nLtNbT2015} for equations related to \eqref{E:MODELWAVE}.
	They \emph{suggest} that in the case $P=1$, it might not be possible to continue the solutions from our main results
	as Sobolev-class solutions in a spacetime neighborhood of a point at which $1 + \Psi$ vanishes. 
	Perhaps this is not surprising since, for the solutions under study, 
	the case $P=1$ corresponds to equation \eqref{E:MODELWAVE} 
	changing from hyperbolic to elliptic type past the degeneracy
	(at least for $C^1$ solutions).
	Specifically, those authors proved a type of Hadamard ill-posedness
	for certain initial data for
	a class of quasilinear first-order systems in $n$ spatial dimensions of the form
	\begin{align} \label{E:QUASILINEARFIRSTORDERSYSTEM}
		\partial_t u
		+ 
		\sum_{a=1}^n A^a(t,x,u) \partial_a u
		& = F(t,x,u),
	\end{align}
	where $(t,x) \in \mathbb{R}^{1+n}$,
	$u$ is a map from $\mathbb{R}^{1+n}$ to $\mathbb{R}^N$
	with $n$ and $N$ arbitrary,
	and the $A^a$ are real $N \times N$ matrices.
	The authors proved several types of results in 
	\cite{nLtNbT2015}, but here we describe only the ones
	that are most relevant for equation \eqref{E:MODELWAVE}. Roughly, in Theorem~1.3 of that paper, 
	for systems of type \eqref{E:QUASILINEARFIRSTORDERSYSTEM} that satisfy some technical conditions, 
	the authors studied perturbations of a background solution, denoted by $\phi = \phi(t,x)$,
	with the following property: the system \eqref{E:QUASILINEARFIRSTORDERSYSTEM} is
	hyperbolic
	when evaluated at $(t,x,u) = (0,x,\mathring{\phi}(x))$, where $\mathring{\phi}(x) := \phi(0,x)$,
	but necessarily becomes elliptic at $(t,x,u) = (t,x,\phi(t,x))$
	at any $t > 0$ due to the branching\footnote{In \cite{nLtNbT2015}, the definition of hyperbolicity is
	that the polynomial (in $\uplambda$) $p := \mbox{\upshape det} \left( \uplambda I - \sum_{a=1}^n \xi_a A^a(t,x,u) \right)$
	should have only real roots, which are eigenvalues of $\sum_{a=1}^n \xi_a A^a(t,x,u)$.
	Moreover, branching roughly means that the eigenvalues are real at $t=0$ but 
	can have non-zero imaginary parts at arbitrarily small values of $t > 0$.} 
	of the eigenvalues of the principal symbol. 
	The assumptions of \cite{nLtNbT2015}*{Theorem~1.3} guarantee that the branching
	is stable under small perturbations. 
	Roughly, for the solutions to \eqref{E:MODELWAVE} from our main results, 
	a similar transition to ellipticity 
	would occur in the case $P=1$ if one were able to classically extend the solution\footnote{As we will explain,
	the solutions from our main results are such that $\Psi$ is strictly decreasing in time at points where $1 + \Psi$ vanishes.} 
	past the time of first vanishing of $1 + \Psi$.
	We now summarize the main aspects of \cite{nLtNbT2015}*{Theorem~1.3}.
	We will use the notation $\mathring{u}$ to denote initial data for the system \eqref{E:QUASILINEARFIRSTORDERSYSTEM}
	and $u$ to denote the corresponding solution (if it exists).
	The theorem is, roughly, as follows: 
	for any $m \in \mathbb{R}$ and $\upalpha \in (1/2,1]$, 
	and any sufficiently small $T > 0$,
	there is no $H^m$-neighborhood
	$\mathcal{U}$ of $\mathring{\phi}$ 
	whose elements launch corresponding solutions
	obeying a bound roughly\footnote{The precise results of \cite{nLtNbT2015}*{Theorem~1.3}
	are localized in space, but here we omit those details for brevity.}
	of the type\footnote{If $f = f(t,x)$, then
	$\| f \|_{W_x^{1,\infty} L_t^{\infty}([0,T])} := \mbox{ess sup}_{t \in [0,T]} \| f(t,\cdot) \|_{W^{1,\infty}}$. \label{FN:MIXEDSOBOLEVNORM}}
	\[
	\displaystyle
	\sup_{\mathring{u} \in \mathcal{U}}
	\frac{\| u - \phi \|_{W_x^{1,\infty} L_t^{\infty}([0,T])}}{\| \mathring{u} - \mathring{\phi} \|_{H^m}^{\upalpha}}
	<
	\infty.
	\]
	Put differently, there exist 
	data arbitrarily close to $\mathring{\phi}$ (as measured by a Sobolev norm of arbitrarily high order)
	such that either the solution does not exist or 
	such that its deviation from $\phi$ becomes arbitrarily large in the low-order norm $\| \cdot \|_{W_x^{1,\infty}}$
	in an arbitrarily short amount of time.
	It would be interesting to determine whether or not a similar result holds for initial data close 
	to that of the data induced by 
	the solutions to \eqref{E:MODELWAVE} from our main results
	at the time of first vanishing of $1 + \Psi$.

	We now discuss the case $P=2$.
	We are not aware of any results for Sobolev-class solutions 
	to quasilinear equations
	that are relevant for extending
	solutions to \eqref{E:MODELWAVE} 
	to exist in a spacetime neighborhood of a point at which $1 + \Psi$ vanishes.
	As we will explain, the main technical difficulty that one encounters is that the solution
	might lose regularity past the degeneracy.
	In the case $P=2$, 
	even though the strict hyperbolicity (see Footnote~\ref{FN:STRICTLYHYPERBOLIC})
	of equation \eqref{E:MODELWAVE} breaks down
	when $1 + \Psi$ vanishes (corresponding to a wave of zero speed),
	the hyperbolicity (see Footnote~\ref{FN:HYPERBOLIC}) of the equation
	nonetheless persists for \emph{all} values of $\Psi$.
	The degeneracy is therefore less severe compared to the case $P=1$ and thus in principle, when $P=2$,
	the Sobolev-class solutions from our main results might be extendable, 
	as a Sobolev-class solution, 
	to a neighborhood of the points where $1 + \Psi$ first vanishes.
	As we alluded to above, 
	the lack of results in this direction might be tied to the following key difficulty: 
	the best energy estimates available for 
	degenerate\footnote{By degenerate, we mean that the wave equation is allowed
	to violate strict hyperbolicity at one or more points.} 
	linear hyperbolic wave equations exhibit a loss of derivatives.
	By this, we roughly mean that 
	the estimates for solutions $\Psi$ to the linear equation 
	are of the form
	$\| \Psi \|_{H^N(\Sigma_t)} \lesssim \| \mathring{\Psi} \|_{H^{N+d}(\Sigma_0)} + \| \mathring{\Psi}_0 \|_{H^{N+d}(\Sigma_0)}$,
	where the loss of derivatives $d$ (relative to the data) depends in a complicated way 
	on the details of the degeneration of the coefficients in the equation;
	see Subsect.\ \ref{SS:OTHERDEGNERATEHYPERBOLICEQUATIONS} for further discussion.
	As is described in \cite{mD1999}, in some cases, the loss of derivatives in the estimates is known to be saturated.
	Since proofs of well-posedness for nonlinear equations typically rely
	on estimates for linearized equations,
	any derivative loss would pose a serious obstacle 
	to extending (in the case $P=2$) the solution of equation \eqref{E:MODELWAVE} 
	as a Sobolev-class solution in a spacetime neighborhood of points at which $1 + \Psi$ vanishes. 
	At the very least, one would need to rely on a method capable of handling a finite 
	loss of derivatives in solutions to quasilinear equations.
	As is well-known \cite{rH1982}, in some cases, 
	it is sometimes possible to handle a finite loss of derivatives
	using the Nash--Moser framework.
	
	Despite the lack of results concerning extending the solution to \eqref{E:MODELWAVE} 
	as a Sobolev-class solution past
	points at which $1 + \Psi$ vanishes,
	there are constructive results in the class $C^{\infty}$. Specifically,
	in one spatial dimension, 
	Manfrin obtained \cite{rM1996} well-posedness results 
	that, for $C^{\infty}$ initial data, 
	allow one to locally continue the solution
	to equation \eqref{E:MODELWAVE} in the case $P=2$
	to a $C^{\infty}$ solution that exists in 
	a spacetime neighborhood of a point at which $1 + \Psi$ vanishes; see Subsect.\ \ref{SS:OTHERDEGNERATEHYPERBOLICEQUATIONS} 
	for further discussion. Manfrin also derived similar results in more than one spatial dimension in \cite{rM1999},
	again treating the case of $C^{\infty}$ data/solutions.
	We are also aware of a few results \cites{mD1999,qHjxHcsL2003} for quasilinear equations
	in more than one spatial dimension
	in which the authors proved local well-posedness in Sobolev spaces
	for equations featuring a degeneracy related to -- but distinct from -- the one under study here.
	However, the degeneracy in those works was created by a ``prescribed semilinear factor'' 
	rather than a quasilinear-type solution-dependent factor. 
	For this reason, it is not clear that the techniques used in those works are of relevance
	for trying to extend solutions to \eqref{E:MODELWAVE} 
	beyond points where $1 + \Psi$ vanishes;
	see the paragraph below equation \eqref{E:HANSEMILINEARWAVE} for further discussion.
	
	To close this subsubsection, 
	we note that there are various well-posedness results 
	\cites{pdAsS1992,yE1985,yElMmM1986} 
	for degenerate wave equations of
	Kirchhoff type. An example of an equation of this type is
	\begin{align} \label{E:KIRCHHOFFEQUATION}
		- \partial_t^2 \Psi
		+
		F\left(\int_{\Omega} |\nabla \Psi|^2 \, dx \right)
		\Delta \Psi
		& = 0,
	\end{align}
	where $\Omega$ is a bounded open set in $\mathbb{R}^n$ and
	$F = F(s) \geq 0$ satisfies various technical conditions
	(with $F=0$ corresponding to the degeneracy).
	However, it remains open whether or not the techniques
	used in studying Kirchhoff-type equations are of relevance
	for proving local well-posedness for equation
	\eqref{E:MODELWAVE} (in the case $P=2$)
	in regions where $1 + \Psi$ is allowed to vanish.

\subsection{Remarks on methods used for studying blowup in solutions to semilinear wave equations}
\label{SS:FURTHERDISCUSSIONANDMOTIVATION}
Although we are not aware of any other results in the spirit of the present work,
there are many results exhibiting the most well-known type of degeneracy that can occur in solutions 
to wave equations in three spatial dimensions: the finite-time \emph{blowup} of initially smooth solutions.
Our main goal in this subsection is to recall some of the most important of these results
but, at the same time, to describe some limitations of the proof techniques
for the study of more general equations. 
We will focus only on constructive\footnote{Constructive proofs of blowup stand, of course, in contrast to proofs of breakdown by contradiction. There are many examples in the literature of proofs of blowup by contradiction for wave or wave-like equations.
Two of the most important ones are
Sideris' blowup result \cite{tS1985}  
(proved by virial identity arguments)
for the compressible Euler equations under a polytropic equation of state
and John's proof \cite{fJ1981} of breakdown for several
classes of semilinear and quasilinear wave equations in three
spatial dimensions.
See also Levine's influential work \cite{hL1974}, in which he proved 
a non-constructive blowup result for semilinear wave equations on an abstract Hilbert space.}
results, by which we mean that the proofs provide a detailed description 
of the degeneracy formation and the mechanisms driving it, as in the present work.
Constructive results, especially those proved via robust techniques,
are clearly desirable if one aims to understand
the mechanisms of breakdown in solutions to realistic physical and geometric systems.
They are also important if one aims to continue the solution past the breakdown,
as is sometimes possible if it is not too severe;
see, for example, \cite{dCaL2016} for a recent result in spherical symmetry 
concerning weakly locally extending solutions to the relativistic Euler equations
past the first shock singularity.
Importantly, we will confine our discussion to prior results for semilinear equations since,
as we mentioned earlier, aside from the shock formation results described at the end of Subsect.\ \ref{SS:OTHERDEGNERATEHYPERBOLICEQUATIONS},
most constructive breakdown results for wave equations in three spatial dimensions
are blowup results for semilinear equations. 

Specifically, most constructive breakdown results for wave equations in three spatial dimensions
are blowup results for semilinear equations (or systems) of the form
$\square_m \Psi = f(\Psi,\partial \Psi)$,
where 
$f$ is a smooth nonlinear term.
Many important\footnote{Arguably, 
the most sophisticated blowup results
of this type have been proved for nonlinearities that correspond to
energy critical equations.} 
approaches have been developed to prove constructive 
blowup for such equations,
especially for scalar equations with 
$f = f(\Psi)$ given by a power law
and for systems of wave-map type;
see, for example,
\cites{ceKfM2008,rD2010,rDbS2014,rDbS2012,jKwS2014,jKwSdT2009,rDmHjKwS2014,iRjS2010,
jKwSdD2008,pRiR2012,tDcKfM2012,yMfMpRjS2014,rDjK2013}.
There are also related results that are conditional in the sense that
they do not guarantee that the solution will blow up.
Instead they characterize the possible behaviors of the solution
by providing information such as
\textbf{i)} how the singularity would form
if the solution is not global and 
\textbf{ii)} the structures 
of the data sets that lead to the various outcomes;
see, for example, \cites{lPdS1975,mS2003,kNwS2011a,kNwS2011b,kNwS2012a,kNwS2012b,jKkNwS2013,jKkNwS2013b,jKkNwS2014,rKbSmV2014,jKkNwS2015}.

Although the above results and others like them
have yielded major advancements in our understanding
of the blowup of solutions to semilinear equations, 
their proofs fundamentally rely 
on tools that are not typically applicable to quasilinear equations.
Here are some important examples, where for brevity, we are not specific about exactly which semilinear
equations have been treated with the stated technique:
\begin{itemize} 
\item
The existence of a conserved energy 
(which is not available for some important quasilinear equations, such as Einstein's equations\footnote{For asymptotically flat solutions to Einstein's equations, 
the ADM mass is conserved. However, in three spatial dimensions without symmetry assumptions, 
this quantity has thus far proven to be too weak to be of any use in controlling solutions.}).
This allows, among other things, for the application of techniques from Hamiltonian mechanics.
\item The invariance of the solutions under appropriate re-scalings 
(which is not a feature of some important quasilinear equations, such as the compressible Euler equations\footnote{In particular, 
the fluid equation of state does not generally enjoy any useful scaling transformation properties.}).
\item The availability of well-posedness results in low regularity spaces
such as the energy space
(Lindblad showed \cite{hL1998} that low regularity well-posedness
fails for a large class of quasilinear equations in three spatial dimensions).
\item The existence of a non-trivial ground state solution
(corresponding to the existence of a soliton solution)
and sharp classification results for the possible behaviors of the solution
for initial data with energy less than the ground state:
either there is finite-time blowup in both time directions 
or global existence, according to the sign of a functional
(for quasilinear equations, there is no known analog of this kind of dichotomy).
Moreover, in some cases, there are more complicated classification results available 
for solutions with energy just above the ground state.
\item A characterization of a certain norm of the ground state 
as a size threshold separating global scattering solutions from ones that can blow up
or exhibit other degenerate behavior
(again, for quasilinear equations, there is no known analog of this kind of dichotomy).
\item A characterization of the ground state as the universal blowup-profile
	under various assumptions.
\item  The availability of profile decompositions for sequences bounded in the natural
	energy space, which allows one to view the sequence as a superposition of
	linear solutions plus a small error
	(for quasilinear equations, there is no known analog of this).
\item Channel of energy type arguments showing that a portion of the solution
propagates precisely at speed one
(again, for quasilinear equations, there is no known\footnote{It is conceivable that channel of energy type results
might hold for certain quasilinear wave equations in various solution regimes,
since channel of energy type arguments seem to be somewhat stable under perturbations.} 
analog of this phenomenon).
\item The possibility of sharply characterizing the spectrum, see for example \cite{oCmHwS2012},
	of linear operators tied to the dynamics
	(which, for quasilinear equations in many solution regimes, is exceedingly difficult).
\end{itemize}

Although the above methods are impressively powerful within their domain of applicability, 
since they do not seem to apply to quasilinear equations,
we believe that it is important to develop new methods
for studying the kinds of breakdown that can occur in the quasilinear case.
It is for this reason that we have chosen
to study the model wave equations \eqref{E:MODELWAVE}.

\subsection{Brief overview of the analysis}
\label{SS:BRIEFOVERVIEW}
As we mentioned earlier, the solutions that we study are such that $\mathring{\Psi}$,
$\nabla \mathring{\Psi}_0$ (where $\nabla$ denotes the spatial coordinate gradient), 
and sufficiently many of their spatial derivatives 
are ``nonlinearly small'' (in appropriate norms)
compared to 
$[\mathring{\Psi}_0]_- 
:= 
|\min \lbrace \mathring{\Psi}_0, 0 \rbrace|
$
and
$
	\displaystyle
	\frac{1}{\| \mathring{\Psi}_0 \|_{L^{\infty}(\Sigma_0)}}
$.
A key aspect of our work is that we are able to propagate the smallness,
long enough for the coefficient $1 + \Psi$ in equation \eqref{E:MODELWAVE}
to vanish. Put differently, our main results show that under the smallness assumptions,
the solution to \eqref{E:MODELWAVE} behaves in many ways
like a solution to the second-order ODE 
$
\displaystyle
\frac{d^2}{dt^2} \Psi = 0
$. The reason that $\Psi$ vanishes for the first time is that
$\partial_t \Psi$ is sufficiently negative at one or more spatial points,
a condition that persists by the previous remarks.
To control solutions, we use the (non-conserved) energy\footnote{Throughout, $d \underline{x} := dx^1 dx^2 dx^3$ 
denotes the standard Euclidean volume form on $\Sigma_t$.}
\begin{align} \label{E:INTROENERGYDEF}
		\mathcal{E}_{[2,5]}(t)
		& :=
			\sum_{k'=2}^5
			\int_{\Sigma_t}
				|\partial_t \nabla^{k'} \Psi|^2
				+
				(1 + \Psi)^P |\nabla \nabla^{k'} \Psi|^2
				+ 
				|\nabla^{k'} \Psi|^2
		\, d \underline{x}.
	\end{align}
We avoid using low-order energies corresponding to
$k'=0,1$ in \eqref{E:INTROENERGYDEF} because 
for the solution regime under consideration
such energies would contain terms that are allowed to be large,
and we prefer to work only with small energies.
Hence, to control the low-order derivatives 
of $\Psi$, we derive ODE-type estimates
that rely in part on the energy estimates for its higher derivatives and Sobolev embedding.
Analytically, the main challenge is that the vanishing of $1 + \Psi$ leads to the degeneracy of 
the top-order spatial derivative terms in \eqref{E:INTROENERGYDEF}, 
which makes it difficult to control some top-order error integrals
in the energy estimates. 

To close the energy estimates, 
we exploit the following monotonicity, 
which is available due to our assumptions on the data:
\begin{quote}
$\partial_t \Psi$ 
is \underline{quantitatively strictly negative}
in a neighborhood of points where $1 + \Psi$ is close to $0$. 
\end{quote}
This quantitative negativity yields,
in our energy identities,
the spacetime error integral 
\begin{align} \label{E:FRICTIONINTEGRAL}
\int_{s=0}^t
\int_{\Sigma_s}
	(\partial_t \Psi) (1 + \Psi)^{P-1} |\nabla \nabla^{k'} \Psi|^2
\, d \underline{x}
\, ds,
\end{align}
which has a ``friction-type'' sign in regions where 
$1 + \Psi$ is close to $0$ but positive;
see the spacetime integral on the left-hand side of \eqref{E:MAININTEGRALINEQUALITYFORENEGY}.
It turns out that the availability of this spacetime integral 
compensates for the degeneracy of the
energy \eqref{E:INTROENERGYDEF}
and yields integrated control over the spatial derivatives up to top-order;
\textbf{this is the key to closing the proof}.

\subsection{Comparing with and contrasting against other results for degenerate hyperbolic equations}
\label{SS:OTHERDEGNERATEHYPERBOLICEQUATIONS}
For solutions such that $1 + \Psi$ is near $0$,
equation \eqref{E:MODELWAVE} can be viewed as a ``nearly degenerate'' quasilinear 
hyperbolic PDE. For this reason, the 
proofs of our main results have ties to some prior results on 
degenerate hyperbolic PDEs, which we now discuss.
In one spatial dimension,
various aspects of degenerate hyperbolic PDEs have been explored 
in the literature, such as the branching of singularities
\cites{sA1978,kAgN1981,kAgN1982,kAgN1983,kAgN1984},
uniqueness of solutions for equations
that are hyperbolic in one region but that can change type \cite{mRmR2016},
and conditions that are \emph{necessary} for well-posedness
\cite{kY1989}. However, in the rest of this subsection,
we will discuss only well-posedness results since they are 
most relevant in the context of our main results. 

In one spatial dimension, there
are many results on well-posedness, in various function spaces, 
for degenerate linear wave equations for the form
\begin{align} \label{E:MODELLINEARDEGENERATEHYPERBOLIC}
	-\partial_t^2 \Psi
	+ a(t,x) \partial_x^2 \Psi
	+ b(t,x) \partial_x \Psi
	+ c(t,x) \partial_t \Psi
	= f(t,x),
\end{align}
where $a(t,x) \geq 0$ and $a(t,x) = 0$ 
corresponds to degeneracy via a breakdown of strict hyperbolicity.
For example, if the coefficients 
$a(t,x)$,
$b(t,x)$,
and 
$c(t,x)$
are \emph{analytic} and satisfy certain technical assumptions, then
it is known \cite{tN1984}
that equation \eqref{E:MODELLINEARDEGENERATEHYPERBOLIC} 
is well-posed for $C^{\infty}$ data;
see also \cite{tN1980} for similar results.
There are also results on well-posedness for
degenerate semilinear equations.
For example, in \cite{pDpT2001},
the authors used a Nash--Moser argument
to prove $C^{\infty}$ local well-posedness
for semilinear equations of the form
\begin{align} \label{E:NASHMOSERWELLPOSEDNESS}
-\partial_t^2 \Psi
+ a(t,x) \partial_x^2 \Psi
= f(t,x,\Psi,\partial_t \Psi,\partial_x \Psi),
\end{align}
where $a(t,x) \geq 0$ is analytic and $a$ and $f$ satisfy 
appropriate technical assumptions.
We clarify that in contrast to our work here, 
in the above works, the authors 
solved the equation in a spacetime neighborhood 
of points at which the degeneracy occurs.

A serious limitation of the above results is that 
techniques relying on analyticity assumptions 
are of little use for studying 
quasilinear Cauchy problems with Sobolev-class data, 
such as the problems we consider here.
Fortunately, well-posedness results for degenerate linear equations 
that do not rely on analyticity assumptions are also known;
see, for example, \cites{oO1970,kTyT1980,pD1994,qHjHcL2006,tH2012,tHmRkY2013,qHyL2015}.
We note in particular that the results of \cites{oO1970,kTyT1980,qHjHcL2006,tH2012,tHmRkY2013,qHyL2015}
provide Sobolev estimates for the solution in terms of a Sobolev norm of the data, with a finite
loss of derivatives. We also mention the related works \cites{aA2006,aA2007}
(see also the references within), in 
which the author proves well-posedness results 
(in $C^{\infty}$ and Gevrey spaces)
for linear wave equations
with two kinds of degeneracies: \textbf{i)} the breakdown of strict hyperbolicity 
(corresponding to the vanishing of certain coefficients) 
and \textbf{ii)} the blowup of the time derivatives of certain coefficients
in the wave equation. We also mention the works
\cites{vI1975,hIkY2002}, in which the authors obtain necessary and sufficient
conditions for the Gevrey space well-posedness of degenerate linear hyperbolic equations.

Most relevant for our work here is
Manfrin's aforementioned proof \cites{rM1996}
of $C^{\infty}$ well-posedness for various degenerate \emph{quasilinear} wave equations in one spatial dimension
(see also \cite{rM1999} for a similar result in more than one spatial dimension and the related work \cite{cBrM2000}),
including those of the form
\begin{align} \label{E:MANFRINDEGENERATEHYPERBOLIC}
	-\partial_t^2 \Psi
	+ \Psi^{2k} a(t,x,\Psi) \partial_x^2 \Psi
	= f(t,x,\Psi),
\end{align}
where $k \geq 1$ is an integer and $a(t,x,\Psi)$ is uniformly bounded
from above and from below, strictly away from $0$
(and $\Psi = 0$ corresponds to the degeneracy).
More precisely, for $C^{\infty}$ initial data, Manfrin used
weighted energy estimates and Nash--Moser estimates
to prove local well-posedness for solutions to \eqref{E:MANFRINDEGENERATEHYPERBOLIC}.
The energy estimate weights are complicated to construct and are based
on dividing spacetime into various regions with the help of ``separating functions''
adapted to the degeneracy.
Note that Manfrin's results apply to our model equation\footnote{More precisely, the role of ``$1 + \Psi$'' in equation
\eqref{E:MODELWAVE} is played by ``$\Psi$'' in equation \eqref{E:MANFRINDEGENERATEHYPERBOLIC}.}
\eqref{E:MODELWAVE} in the case $P=2$.
However, it is an open problem whether or not his results can be extended to 
yield a local well-posedness result for equation \eqref{E:MANFRINDEGENERATEHYPERBOLIC} with data
in Sobolev spaces.

To further explain these results and their connection to our work here,
we consider the simple Tricomi-type equation
\begin{align} \label{E:TRICOMIMODEL}
		- \partial_t^2 \Psi
		+ 
		a(t) \Delta \Psi
		= 0,
\end{align}
where $a(t) \geq 0$. 
It is known \cite{fCsS1982} that in one spatial dimension,
the linear equation \eqref{E:TRICOMIMODEL} can be \emph{ill-posed},\footnote{In \cite{fCsS1982},
which addressed solutions in one spatial dimension,
the authors exhibited a smooth function $a(t) \geq 0$ with $a(t_0) = 0$ for some $t_0 > 0$
and data such that there is no distributional solution 
to \eqref{E:TRICOMIMODEL} with the given data that extends past time $t_0$.}
even if $a = a(t)$ is $C^{\infty}$.
Hence, it should not be taken for granted that we can 
(for suitable data) 
solve equation \eqref{E:MODELWAVE} in Sobolev spaces all the way up to the time of first vanishing of
$1 + \Psi$.
Roughly, what can go wrong in an attempt to solve the linear equation \eqref{E:TRICOMIMODEL}
is that $a(t)$ can be highly oscillatory 
near a point $t_0$ with $a(t_0) = 0$. In fact, in the example
from \cite{fCsS1982}, $a(t)$ oscillates \emph{infinitely many} 
times near $t_0$. This generates, in the basic energy identity, 
an uncontrollable term involving the ratio
$
\displaystyle
\frac{\frac{d}{dt} a(t)}{a(t)}
$ 
and leads to ill-posedness in domains $[A,B) \times \mathbb{R}$
when $t_0 \in [A,B)$.

In all of the aforementioned well-posedness results, 
the technical conditions imposed on the coefficients
rule out the infinite oscillatory behavior 
from \cite{fCsS1982} that led to ill-posedness.
To provide a more concrete example, we note that in one spatial dimension, 
Han derived \cite{qH2010}
degenerate energy estimates 
for linear wave equations of the form
\begin{align} \label{E:HANSEMILINEARWAVE}
	- \partial_t^2 \Psi
	+ a(t,x) \partial_x^2 \Psi
	+ b_0(t,x) \partial_t \Psi
	+ b(t,x) \partial_x \Psi
	+ c(t,x) \Psi
	= f(t,x),
\end{align}
where the coefficients satisfy certain technical conditions,
including, roughly speaking, that $a(t,x) \geq 0$
behaves like 
$t^m + c_{m-1}(x) t^{m-1} + \cdots + c_1(x) t + c_0(x)$.
In particular, even though $a$ is allowed to vanish at some points,
it does not exhibit highly oscillatory behavior
in the $t$ direction.
In \cite{qHjHcL2006}, similar results were derived in $n \geq 1$
spatial dimensions. 

We now describe Dreher's PhD thesis \cite{mD1999}, which involves the study of equations
that share some common features with equation \eqref{E:MODELWAVE} 
near the degeneracy $1 + \Psi = 0$.
Specifically, in his thesis, Dreher proved local well-posedness results in Sobolev spaces
for several classes
of quasilinear hyperbolic PDEs in any number of dimensions
with various kinds of space and time
degeneracies. However,
a key difference between the equations studied 
by Dreher in his thesis and our work
is that the degeneracies in \cite{mD1999} were ``prescribed''
in the sense that they were caused only by degenerate semilinear factors
that explicitly depend on the time and space variables.
That is, if one deletes the degenerate semilinear factors, then one obtains
a strictly hyperbolic PDE for which local well-posedness follows from standard techniques.
Dreher made technical assumptions on the degenerate semilinear factors
that were sufficient for proving well-posedness.
In contrast, the degeneracy caused by $1 + \Psi = 0$ in equation \eqref{E:MODELWAVE} is 
``purely quasilinear'' in nature.
The following model equation in one spatial dimension
gives a sense of the kinds of prescribed degeneracy
treated by Dreher in \cite{mD1999}:
\begin{align} \label{E:QUASILINEARWITHDEGENERATESEMILINEARWEIGHT}
	- \partial_t^2 \Psi
	+ t^2 f(\Psi) \partial_x^2 \Psi
	& = 0,
\end{align}
where $f$ is smooth and satisfies $f(\Psi) > 0.$ We stress that the absence of strict hyperbolicity 
in a neighborhood of $\Sigma_0$ is \emph{not} caused by the quasilinear factor $f(\Psi)$,
but rather by the semilinear factor $t^2$.
A related example, coming from geometry, 
is the aforementioned work \cite{qHjxHcsL2003},
in which the authors proved the existence of
local $C^k$ embeddings of surfaces 
of non-negative Gaussian curvature
into $\mathbb{R}^3$. 
The quasilinear system of PDEs studied there degenerated 
at points where the Gauss curvature of the surface vanishes.
As in Dreher's work \cite{mD1999},
the degeneracy was ``prescribed''
in the sense that it was caused by the Gauss curvature
(which is ``known'').
Hence, the authors were free to make technical assumptions 
on the Gauss curvature to ensure the local well-posedness of the PDE system.

We now give another example of prior work that is closely connected to our main results.
In \cite{zRiWhY2016}, the authors proved local well-posedness 
in homogeneous Sobolev spaces 
on domains of the form $[0,T) \times \mathbb{R}^n$
for semilinear Tricomi equations of the form
\begin{align} \label{E:SEMILINEARTICOMI}
	-\partial_t^2 \Psi + t^P \Delta \Psi 
	& = f(\Psi),
\end{align}
where $P \in \mathbb{N}$
and $f$ is a nonlinearity such that $f$ and $f'$ obey
certain $P,n-$dependent power-law growth bounds at $\infty$.
See \cites{zRiWhY2015a,zRiWhY2015b} for related results.
Note that the coefficient $t^P$ in \eqref{E:SEMILINEARTICOMI}
does not oscillate; once again, this is the key difference 
compared to the ill-posedness result
for equation \eqref{E:TRICOMIMODEL} mentioned above.
As we described in Subsect.\ \ref{SS:BRIEFOVERVIEW},
equation \eqref{E:SEMILINEARTICOMI} 
is a good model for the solutions to equation \eqref{E:MODELWAVE}
provided by our main results
in the sense that the degenerating coefficient $(1 + \Psi)^P$ in \eqref{E:MODELWAVE}
behaves in some ways, when $1 + \Psi$ is small, 
like\footnote{The key point is that since our solutions are such that $\partial_t \Psi < 0$ when $1 + \Psi = 0$,
it follows that $1 + \Psi$ behaves, to first order, linearly in $t$ near points where it vanishes.} 
the coefficient $t^P$ (near $t = 0$)
in \eqref{E:SEMILINEARTICOMI}.

In view of the above discussion, 
we believe that one should not expect to be able to solve equation
\eqref{E:MODELWAVE} in Sobolev spaces all the way up to points with $1 + \Psi = 0$
unless one makes assumptions on the data 
that preclude highly oscillatory behavior in regions where $1 + \Psi$ is small.
In this article, we avoid the oscillatory behavior by
exploiting the \emph{relative} largeness of $[\partial_t \Psi]_-$ 
and the \emph{relative} smallness of $\partial_t^2 \Psi$
in regions where $1 + \Psi$ is small, 
which are present at time $0$ and which we propagate;
see the estimates \eqref{E:PSISMALLIMPLIESPARTIALTPSINEGATIVE} 
and \eqref{E:BOOTSTRAPIMPROVED}.
As we have mentioned,
the relative largeness of $[\partial_t \Psi]_-$
can be viewed as a kind of monotonicity in the problem.
One might say that this monotonicity makes up for the lack of remarkable
structure in \eqref{E:MODELWAVE}, including that it is not an Euler-Lagrange equation, 
its solutions admit no known coercive conserved quantities,
and the nonlinearities are not signed.
As we described in Subsect.\ \ref{SS:BRIEFOVERVIEW},
this monotonicity yields an important signed spacetime
integral that we use to close the energy estimates;
see the spacetime integral on the left-hand side of \eqref{E:MAINAPRIORIENERGYESTIMATE}.
The largeness of $[\partial_t \Psi]_-$
is connected to so-called \emph{Levi-type} conditions
that have appeared in the literature.
Roughly, a Levi condition is a quantitative relationship between the
sizes of various coefficients in the equation and their derivatives.
As an example, we note that in their 
study \cite{pDpT2001} of well-posedness 
for equation \eqref{E:NASHMOSERWELLPOSEDNESS} with analytic coefficients,
the authors studied linearized equations
of the form \eqref{E:MODELLINEARDEGENERATEHYPERBOLIC}
under the Levi condition
$|b(t,x)| \lesssim |a(t,x)| + |\partial_t \sqrt{a}(t,x)|$;
the Levi condition allowed them,
for the linearized equation, 
to construct suitable weights
for the energy estimates (even at points where $a$ vanishes), 
which were sufficient for proving well-posedness.
In the problems under study here, 
the largeness of $[\partial_t \Psi]_-$
at points with $1 + \Psi = 0$
can be viewed as a Levi-type condition
for the coefficient $(1 + \Psi)^P$
in equation \eqref{E:MODELWAVE},
which allows us to control 
various error terms
that arise when we derive energy estimates
for the solution's higher derivatives.

The degenerate energy estimates featured 
in our proofs have some features in common with 
the foundational works
\cites{sA1999a,sA1999b,sA2001b,sA2002,dC2007}
of Alinhac and Christodoulou
on the formation of shock singularities in solutions to quasilinear wave equations in 
two or three spatial dimensions;
see also the follow-up works \cites{dCsM2014,jSgHjLwW2016,jS2016,bDiWhY2017,bDiWhY2015b,bDiWhY2015} 
and the survey article \cite{gHsKjSwW2016}.
In those works, the authors constructed a dynamic geometric coordinate system that
degenerated in a precise fashion\footnote{In essence, the authors
straightened out the characteristics via a solution-dependent change of coordinates.} 
as the shock formed. Consequently, relative to the geometric coordinates,
the solution remains rather smooth,\footnote{The high-order geometric energies
were allowed to blow up at the shock, which led to enormous technical complications in the proofs.
Note that this possible high-order energy blowup is distinct from the formation of the shock singularity,
which corresponds to the blowup of a low-order Cartesian coordinate partial derivative of the solution.} 
which was a key fact used to control error terms.
A crucial feature of the proofs is that the energy estimates\footnote{There are 
many shock-formation results for solutions to quasilinear equations 
in one spatial dimension,
with important contributions coming from
Riemann \cite{bR1860}, 
Ole{\u{\i}}nik \cite{oO1957}, 
Lax \cite{pL1957}, 
Klainerman and Majda \cite{sKaM1980},
John \cites{fJ1974,fJ1981,fJ1984},
and many others.
However, those results are based exclusively on the
method of characteristics and hence,
unlike in the case of two or more spatial dimensions,
the proofs do not rely on energy estimates.}  
contained weights
that vanished at the shock, which is in analogy with the vanishing of
the weight $(1 + \Psi)^P$ in \eqref{E:INTROENERGYDEF}
at the degeneracy.
A second crucial feature of the proofs of shock formation
is that they relied on the fact that the weight has a \emph{quantitatively strictly negative time derivative}
in a neighborhood of points where it vanishes. This yields a critically important monotonic spacetime integral
that is in analogy with the one \eqref{E:FRICTIONINTEGRAL} that we use to control various error terms 
in the present work.

We close this subsection by noting that 
the degeneracy that we encounter in our study of equation \eqref{E:MODELWAVE}
is related to -- but distinct from -- a particular kind of absence of strict hyperbolicity
that has been studied in the context of the compressible Euler equations for initial data
satisfying the physical vacuum condition;
see, for example, \cites{dChLsS2010,dCsS2011,dCsS2012,jJnM2011,jJnM2009}.
The key difference between those works and ours is that in 
those works, the degeneracies occurred along
the fluid-vacuum boundary in spacelike directions rather than a timelike one. 
In particular, the degeneracy was already present at time $0$.
More precisely, at time $0$, the fluid density vanished at a certain rate,
meaning that the density's derivative in the (spacelike) normal direction
to the vacuum boundary satisfied a quantitative signed condition.
It turns out that this condition yields a signed integral in the
energy identities that is essential for closing the energy estimates. 
The signed integral exploited in those works is analogous to the integral 
\eqref{E:FRICTIONINTEGRAL}, but the integrals in the above papers 
were available because the solution's (spacelike) normal derivative had a sign, which is in contrast
to the sign of the timelike derivative $\partial_t \Psi$ exploited in the present work.
With the help of the signed integral, the authors of the above papers
were able to prove degenerate energy estimates with weights that vanished at a certain rate in the normal direction
to the vacuum boundary. Ultimately, these degenerate estimates 
allowed them to prove local well-posedness in Sobolev spaces with weights
that degenerate at the fluid-vacuum boundary.

\subsection{Notation}
\label{SS:NOTATION}
In this subsection, we summarize some notation that we use throughout.

\begin{itemize}
	\item $\lbrace x^{\alpha} \rbrace_{\alpha=0,1,2,3}$ denotes
		the standard rectangular coordinates on $\mathbb{R}^{1+3} = \mathbb{R} \times \mathbb{R}^3$
		and
			$
	\displaystyle
	\partial_{\alpha} 
	:=
	\frac{\partial}{\partial x^{\alpha}}
	$
	denotes the corresponding coordinate partial derivative vector fields.
		$x^0 \in \mathbb{R}$ is the time coordinate and $\underline{x} := (x^1,x^2,x^3) \in \mathbb{R}^3$
		are the spatial coordinates.
	\item We often use the alternate notation $x^0 = t$ and $\partial_0 = \partial_t$.
	\item Greek ``spacetime'' indices such as $\alpha$ vary over $0,1,2,3$ and
		Latin ``spatial'' indices such as $a$ vary over $1,2,3$.
		We use primed indices, such as $a'$, in the same way that we use their non-primed counterparts.
		We use Einstein's summation convention in that repeated indices are summed
		over their respective ranges.
	\item We raise and lower indices with $g^{-1}$ and $g$ respectively
		(\emph{not} with the Minkowski metric!).
	\item We sometimes omit the arguments of functions appearing in pointwise inequalities. For example,
		we sometimes write $|f| \leq C \mathring{\upepsilon}$
		instead of $|f(t,\underline{x})| \leq C \mathring{\upepsilon}$.
	\item $\nabla^k \Psi$ denotes the array comprising all $k^{th}-$order
		derivatives of $\Psi$ with respect to the rectangular spatial coordinate vector fields.
		We often use the alternate notation $\nabla \Psi$ in place of $\nabla^1 \Psi$.
		For example, $\nabla^1 \Psi = \nabla \Psi := (\partial_1 \Psi, \partial_2 \Psi, \partial_3 \Psi)$.
	\item 
	$|\nabla^{\leq k} \Psi| 
	:= \sum_{k'=0}^k |\nabla^{k'} \Psi|
	$.
	\item 
		$|\nabla^{[a,b]} \Psi| 
		:= \sum_{k'=a}^b |\nabla^{k'} \Psi|
		$.
	\item $H^N(\Sigma_t)$ denotes the standard Sobolev space of functions on $\Sigma_t$
	with corresponding norm 
	\[
	\displaystyle
	\| f \|_{H^N(\Sigma_t)}
	:= 
	\left\lbrace
		\sum_{a_1 + a_2 + a_3 \leq N}
		\int_{\underline{x} \in \mathbb{R}^3}
			|\partial_1^{a_1} \partial_2^{a_2} \partial_3^{a_3}f(t,\underline{x})|^2
		\, d \underline{x}
	\right\rbrace^{1/2}.
	\]
	In the case $N=0$, we use the notation ``$L^2$'' in place of ``$H^0$.''
	\item 
	$L^{\infty}(\Sigma_t)$ denotes the standard Lebesgue space of functions on $\Sigma_t$
	with corresponding norm 
	$
	\displaystyle
	\| f \|_{L^{\infty}(\Sigma_t)}
	:= 
	\mbox{\upshape ess sup}_{\underline{x} \in \mathbb{R}^3}
	|f(t,\underline{x})|
	$.
	\item If $A$ and $B$ are two quantities, then we often write 
		$A \lesssim B$
		to indicate that ``there exists a constant $C > 0$ such that $A \leq C B$.''
	\item We sometimes write $\mathcal{O}(B)$ to denote a quantity $A$ 
	with the following property: there exists a constant $C > 0$ such that $|A| \leq C |B|$.
\end{itemize}

\section{Assumptions on the initial data and bootstrap assumptions}
\label{S:DATAANDBOOTSTRAP}
In this short section, we state our assumptions on the data
$
(\Psi|_{\Sigma_0},\partial_t \Psi|_{\Sigma_0})
:= (\mathring{\Psi},\mathring{\Psi}_0)
$
for the model equation \eqref{E:MODELWAVE}
and formulate bootstrap assumptions that are convenient
for studying the solution.
We also show that there exist data satisfying 
the assumptions.

\subsection{Assumptions on the data}
\label{SS:DATAASSUMPTIONS}
We assume that the initial data 
are compactly supported 
and satisfy the assumptions
\begin{align} \label{E:DATASIZE}
	\| \nabla^{\leq 4} \mathring{\Psi} \|_{L^{\infty}(\Sigma_0)}
	+
	\| \nabla^{[1,3]} \mathring{\Psi}_0 \|_{L^{\infty}(\Sigma_0)}
	+
	\| \nabla^2 \mathring{\Psi} \|_{H^4(\Sigma_0)}
	+
	\| \nabla^2 \mathring{\Psi}_0 \|_{H^3(\Sigma_0)}
	& \leq \mathring{\upepsilon},
	\qquad
	\| \mathring{\Psi}_0 \|_{L^{\infty}(\Sigma_0)}
	\leq \mathring{\updelta},
\end{align}	
where $\mathring{\upepsilon}$ 
and $\mathring{\updelta}$
are two data-size parameters that 
we will discuss below
(roughly, $\mathring{\upepsilon}$ will have to be small for our proofs to close).
Roughly speaking, in our analysis, we will approximately propagate the above
size assumptions all the way up until the time of breakdown in hyperbolicity,
except at the top derivative level.
More precisely, we are not able to uniformly control the
top-order spatial derivatives of $\Psi$ in the norm 
$\| \cdot \|_{L^2(\Sigma_t)}$ up to the time of breakdown
due to the presence of degenerate weights in our energy
(see Def.\ \ref{D:ENERGYDEF}).

Before we can proceed, we must first introduce the 
crucial parameter $\mathring{\updelta}_*$ that 
controls the time of first breakdown in hyperbolicity;
our analysis shows that for
$\mathring{\upepsilon}$ sufficiently small, the time
of first breakdown is 
$\left\lbrace 1 + \mathcal{O}(\mathring{\upepsilon})\right\rbrace \mathring{\updelta}_*^{-1}$;
see also Remark~\ref{R:CRUCIALDELTAPARAMETER}.

\begin{definition}[\textbf{The parameter that controls the time of breakdown in hyperbolicity}]
	\label{D:CRUCIALDATASIZEPARAMETER}
	We define the data-dependent parameter $\mathring{\updelta}_*$ as 
	\begin{align} \label{E:CRUCIALDATASIZEPARAMETER}
		\mathring{\updelta}_*
		& := 
		\max_{\Sigma_0} [\mathring{\Psi}_0]_-.
  \end{align}
\end{definition}

\begin{remark}[\textbf{The relevance of $\mathring{\updelta}_*$}]
\label{R:CRUCIALDELTAPARAMETER}
The solutions that we study are such that\footnote{Here ``$A \sim B$'' imprecisely indicates
that $A$ is well-approximated by $B$.}
$\mathring{\Psi} \sim 0$
and
$
\partial_t^2 \Psi \sim 0
$
(throughout the evolution).
Hence, by the fundamental theorem of calculus,
we have
$\partial_t \Psi(t,\underline{x}) \sim \mathring{\Psi}_0(\underline{x})$
and $1 + \Psi(t,\underline{x}) \sim 1 + t \mathring{\Psi}_0(\underline{x})$. 
From this last expression, we see that
$1 + \Psi$ is expected to vanish for the first time at approximately
$t = \mathring{\updelta}_*^{-1}$.
See Lemma~\ref{L:POINTWISEFORPSIANDPARTIALTPSI} for the precise statements.
\end{remark}

\subsection{Bootstrap assumptions}
\label{SS:BOOTSTRAP}
In proving our main results, we find it convenient to rely on
a set of bootstrap assumptions, which we provide in this subsection.

\medskip

\noindent \underline{\textbf{The size of} $T_{(Boot)}$}.
We assume that $T_{(Boot)}$
is a bootstrap time with
\begin{align} \label{E:BOOTSTRAPTIME}
	0 < T_{(Boot)} \leq 2 \mathring{\updelta}_*^{-1}.
\end{align}
The assumption \eqref{E:BOOTSTRAPTIME} 
gives us a sufficient margin of error
to prove that finite-time degeneration of hyperbolicity occurs,
as we explained in Remark~\ref{R:CRUCIALDELTAPARAMETER}.

\medskip

\noindent \underline{\textbf{Degeneracy has not yet occurred}}.
We assume that for $(t,\underline{x}) \in [0,T_{(Boot)}) \times \mathbb{R}^3$,
we have
\begin{align} \label{E:HYPERBOLICBOOTSTRAP}
	0 < 1 + \Psi (t,\underline{x}).
\end{align}

\medskip

\noindent \underline{$L^{\infty}$ \textbf{bootstrap assumptions}}.
We assume that for $t \in [0,T_{(Boot)})$, we have
\begin{subequations}
\begin{align} \label{E:PSIITSELFBOOTSTRAP}
	\| \Psi \|_{L^{\infty}(\Sigma_t)}
	& \leq 2 \mathring{\updelta}_*^{-1} \mathring{\updelta}	
		+ \varepsilon^{1/2},
			\\
	\| \partial_t \Psi \|_{L^{\infty}(\Sigma_t)}
	&  \leq 
		\mathring{\updelta}	
		+ \varepsilon^{1/2},
			\label{E:PARTIALTPSIBOOTSTRAP} \\
	\| \nabla^{[1,3]} \Psi \|_{L^{\infty}(\Sigma_t)},
		\,
	\| \partial_t \nabla^{[1,3]} \Psi \|_{L^{\infty}(\Sigma_t)},
		\,
	\| \partial_t^2 \nabla^{\leq 1} \Psi \|_{L^{\infty}(\Sigma_t)}
	& \leq \varepsilon,
	 \label{E:SMALLLINFTYBOOTSTRAP}
\end{align}
\end{subequations}
where $\varepsilon > 0$ is a small bootstrap parameter;
we describe our smallness assumptions in the next subsection.

\begin{remark}[\textbf{The solution remains compactly supported in space}]
		From \eqref{E:PSIITSELFBOOTSTRAP}, we deduce that the wave speed
		$(1 + \Psi)^{P/2}$ associated to equation~\eqref{E:MODELWAVE}
		remains uniformly bounded from above for $(t,\underline{x}) \in [0,T_{(Boot)}) \times \mathbb{R}^3$.
		Hence, there exists a large, data-dependent ball $B \subset \mathbb{R}^3$
		such that $\Psi(t,\underline{x})$ vanishes
		for $(t,\underline{x}) \in [0,T_{(Boot)}) \times B^c$,
		where $B^c$ denotes the complement of $B$ in $\mathbb{R}^3$.
\end{remark}

\subsection{Smallness assumptions}
\label{SS:SMALLNESSASSUMPTIONS}
For the rest of the article, 
when we say that ``$A$ is small relative to $B$,''
we mean that $B > 0$ and that there exists a continuous increasing function 
$f :(0,\infty) \rightarrow (0,\infty)$ 
such that 
$
\displaystyle
A \leq f(B)
$.
In principle, the functions $f$ could always be chosen to be 
polynomials with positive coefficients or exponential functions.
However, to avoid lengthening the paper, we typically do not 
specify the form of $f$.

Throughout the rest of the paper, we make the following
relative smallness assumptions. We
continually adjust the required smallness
in order to close our estimates.
\begin{itemize}
	\item The bootstrap parameter $\varepsilon$ from Subsect.\ \ref{SS:BOOTSTRAP}
		is small relative to $\mathring{\updelta}^{-1}$,
		where $\mathring{\updelta}$ is the data-size parameter 
		from \eqref{E:DATASIZE}.
	\item $\varepsilon$ is small relative to 
		the data-size parameter $\mathring{\updelta}_*$ 
		from \eqref{E:CRUCIALDATASIZEPARAMETER}.
\end{itemize}
The first assumption will allow us to control error terms that,
roughly speaking, are of size $\varepsilon \mathring{\updelta}^k$ 
for some integer $k \geq 0$. The second assumption 
is relevant because the expected degeneracy-formation time is 
approximately $\mathring{\updelta}_*^{-1}$
(see Remark~\ref{R:CRUCIALDELTAPARAMETER});
the assumption will allow us to show that various
error products featuring a small factor $\varepsilon$
remain small for $t \leq 2 \mathring{\updelta}_*^{-1}$, 
which is plenty of time for us to show that 
$1 + \Psi$ vanishes.

Finally, we assume that
\begin{align} \label{E:DATAEPSILONVSBOOTSTRAPEPSILON}
	\varepsilon^{3/2}
	& \leq
	\mathring{\upepsilon} 
	\leq \varepsilon,
\end{align}
where $\mathring{\upepsilon}$ is the data smallness parameter from \eqref{E:DATASIZE}.

\subsection{Existence of data}
\label{SS:EXISTENCEOFDATA}
It is easy to construct data
such that the parameters
$\mathring{\upepsilon}$,
$\mathring{\updelta}$,
and 
$\mathring{\updelta}_*$
satisfy the relative size assumptions stated in Subsect.\ \ref{SS:SMALLNESSASSUMPTIONS}.
For example, we can start with \emph{any} smooth compactly supported data 
$(\mathring{\Psi},\mathring{\Psi}_0)$
such that $\min_{\mathbb{R}^3} \mathring{\Psi}_0 < 0$.
We then consider the one-parameter family
\[
\left(
	\leftexp{(\uplambda)}{\mathring{\Psi}}(\underline{x}),
	\leftexp{(\uplambda)}{\mathring{\Psi}_0}(\underline{x})
\right)
:=
\left(
	\uplambda^{-1} \mathring{\Psi}(\underline{x}),
	\mathring{\Psi}_0(\uplambda^{-1} \underline{x})
\right).
\]
One can check that for $\uplambda > 0$ sufficiently large,
all of the size assumptions of
Subsect.\ \ref{SS:SMALLNESSASSUMPTIONS} are satisfied.
The proof relies on the simple scaling identities
\begin{subequations}
\begin{align}
	\nabla^k \leftexp{(\uplambda)}{\mathring{\Psi}}(\underline{x})
	= \uplambda^{-1} (\nabla^k \mathring{\Psi})(\underline{x}),
			\\
	\nabla^k \leftexp{(\uplambda)}{\mathring{\Psi}_0}(\underline{x})
= \uplambda^{-k} (\nabla^k \mathring{\Psi}_0)(\uplambda^{-1} \underline{x})
\end{align}
\end{subequations}
and
\begin{subequations}
\begin{align}
	\left\|
		\nabla^k \leftexp{(\uplambda)}{\mathring{\Psi}}
	\right\|_{L^2(\Sigma_0)}
	& = \uplambda^{-1} \| \mathring{\Psi} \|_{L^2(\Sigma_0)},
			\\
	\left\|
		\nabla^k \leftexp{(\uplambda)}{\mathring{\Psi}_0}
	\right\|_{L^2(\Sigma_0)}
	& = \uplambda^{3/2 - k}
			\| \mathring{\Psi}_0 \|_{L^2(\Sigma_0)}.
\end{align}
\end{subequations}

\begin{remark}[\textbf{Degeneracy occurs for solutions launched by 
	any appropriately rescaled non-trivial data}]
	\label{R:GENERICDEGENERACY}
	The discussion in Subsect.\ \ref{SS:EXISTENCEOFDATA}
	can easily be extended to show that
	if $\mathring{\Psi}_0$ is non-trivial, then
	one \emph{always} generates data to which
	our results apply
	by considering the rescaled data
	$
	\left(
	\leftexp{(\uplambda)}{\mathring{\Psi}},
	\leftexp{(\uplambda)}{\mathring{\Psi}_0}
	\right)
	$
	with $\uplambda$ sufficiently large.
	More precisely, if
	$\min_{\mathbb{R}^3} \mathring{\Psi}_0 = 0$, then
	we must have
	$\max_{\mathbb{R}^3} \mathring{\Psi}_0 > 0$;
	in this case, the degeneracy in the solution generated by the rescaled
	data occurs in the past rather than the future.

\end{remark}

\section{A priori estimates}
\label{S:APRIORIESTIMATES}
In this section, we use the
data-size assumptions and the bootstrap assumptions 
of Sect.\ \ref{S:DATAANDBOOTSTRAP}
to derive a priori estimates for the solution.
This is the main step in the proof our results.

\subsection{Conventions for constants}
	\label{SS:CONVENTIONSFORCONSTANTS}
	In our estimates, the explicit constants $C > 0$ and $c > 0$ are free to vary from line to line.
	\textbf{These constants are allowed to depend on the data-size parameters
	$\mathring{\updelta}$
	and 
	$\mathring{\updelta}_*^{-1}$
	from Subsect.\ \ref{SS:DATAASSUMPTIONS}}.
	However, the constants can be chosen to be 
	independent of the parameters $\mathring{\upepsilon}$ 
	and $\varepsilon$ whenever $\mathring{\upepsilon}$ and $\varepsilon$
	are sufficiently small relative to 
	$\mathring{\updelta}^{-1}$
	and $\mathring{\updelta}_*$
	in the sense described in Subsect.\ \ref{SS:SMALLNESSASSUMPTIONS}.
	For example, under our conventions, we have that $\mathring{\updelta}_*^{-2} \varepsilon = \mathcal{O}(\varepsilon)$.

\subsection{Pointwise estimates}
\label{E:POINTWISEESTIMATES}
In this subsection, we derive pointwise estimates for
$\Psi$ and the inhomogeneous 
terms in the commuted wave equation.

We start with a simple lemma that provides sharp pointwise estimates
for $\Psi$ and $\partial_t \Psi$.

\begin{lemma}[\textbf{Pointwise estimates for} $\Psi$ \textbf{and} $\partial_t \Psi$]
	\label{L:POINTWISEFORPSIANDPARTIALTPSI}
	Under the data-size assumptions of Subsect.\ \ref{SS:DATAASSUMPTIONS},
	the bootstrap assumptions of Subsect.\ \ref{SS:BOOTSTRAP},
	and the smallness assumptions of Subsect.\ \ref{SS:SMALLNESSASSUMPTIONS},
	the following pointwise estimates hold for
	$(t,\underline{x}) \in [0, T_{(Boot)}) \times \mathbb{R}^3$:
	\begin{subequations}
	\begin{align}
		\Psi(t,\underline{x})
		& = t \mathring{\Psi}_0(\underline{x})
			+
			\mathcal{O}(\varepsilon),
			\label{E:PSIWELLAPPROXIMATED} \\
		\partial_t \Psi(t,\underline{x})
		& = 
		\mathring{\Psi}_0(\underline{x})
		+ \mathcal{O}(\varepsilon).
		\label{E:PARTIALTPSIWELLAPPROXIMATED}
	\end{align}
	\end{subequations}
\end{lemma}

\begin{proof}
	To derive \eqref{E:PARTIALTPSIWELLAPPROXIMATED},
	we first use the bootstrap assumptions to
	deduce 
	$
	\left\|
		(1 + \Psi)^P \Delta \Psi
	\right\|_{L^{\infty}(\Sigma_t)}
	\leq C \varepsilon
	$.
	Hence, from equation
	\eqref{E:MODELWAVE},
	we deduce the pointwise bound
	$
	\left|
		\partial_t^2 \Psi
	\right|
	\leq C \varepsilon
	$.
	From this estimate and the fundamental theorem of calculus,
	we conclude the desired bound \eqref{E:PARTIALTPSIWELLAPPROXIMATED}.
	The bound \eqref{E:PSIWELLAPPROXIMATED}
	then follows from the fundamental theorem of calculus,
	\eqref{E:PARTIALTPSIWELLAPPROXIMATED},
	and the small-data bound
	$\| \mathring{\Psi} \|_{L^{\infty}(\Sigma_0)}
	\leq \mathring{\upepsilon}
	\leq \varepsilon
	$.
\end{proof}

The next proposition captures the monotonicity 
that is present at points where $1 + \Psi$ is small.
It is of critical importance for the energy estimates.

\begin{proposition}[\textbf{Monotonicity near the degeneracy}]
	Under the data-size assumptions of Subsect.\ \ref{SS:DATAASSUMPTIONS},
	the bootstrap assumptions of Subsect.\ \ref{SS:BOOTSTRAP},
	and the smallness assumptions of Subsect.\ \ref{SS:SMALLNESSASSUMPTIONS},
	the following statement holds for 
	$(t,\underline{x}) \in [0, T_{(Boot)}) \times \mathbb{R}^3$:
	\begin{align} \label{E:PSISMALLIMPLIESPARTIALTPSINEGATIVE}
		\Psi(t,\underline{x}) 
		\leq
		-
		\frac{1}{2}
		\implies
		\partial_t \Psi(t,\underline{x})
		\leq
		-
		\frac{1}{8}
		\mathring{\updelta}_*,
	\end{align}
	where $\mathring{\updelta}_*$ is the data-dependent parameter from Def.\ \ref{D:CRUCIALDATASIZEPARAMETER}.
	
\end{proposition}

\begin{proof}	
	To prove \eqref{E:PSISMALLIMPLIESPARTIALTPSINEGATIVE},
	we first use the estimates 
	\eqref{E:PSIWELLAPPROXIMATED} and \eqref{E:PARTIALTPSIWELLAPPROXIMATED}
	to deduce that
	$ 
	\Psi(t,\underline{x}) 
	=
	t \partial_t \Psi(t,\underline{x})
	+
	\mathcal{O}(\varepsilon).
	$	
	Hence, if 
	$\Psi(t,\underline{x}) 
		\leq
		-
		\frac{1}{2}
	$,
	then 
	$
	t \partial_t \Psi(t,\underline{x})
	\leq
	- 
	\frac{1}{4}
	$.
	Recalling that $0 \leq t < 2 \mathring{\updelta}_*^{-1}$,
	see \eqref{E:BOOTSTRAPTIME},
	we conclude \eqref{E:PSISMALLIMPLIESPARTIALTPSINEGATIVE}.
	
\end{proof}

We now derive pointwise estimates for the inhomogeneous terms in the commuted wave equation.

\begin{lemma}[\textbf{Pointwise estimates for the inhomogeneous terms}]
	\label{L:POINTWISEESTIMATES}
	Let $\Psi$ be a solution to the wave equation \eqref{E:MODELWAVE}.
	For $k=2,3,4,5$ and $P=1,2$, consider 
	following wave equation,\footnote{We do not bother to state the precise form of $F^{(k)}$ here.}
	obtained by commuting \eqref{E:MODELWAVE} with $\nabla^k$:
	\begin{align} \label{E:NABLAKOMMUTEDWAVE}
		- \partial_t^2 \nabla^k \Psi
		+ 
		(1 + \Psi)^P \Delta \nabla^k \Psi
	& = F^{(k)}.
	\end{align}
	Under the data-size assumptions of Subsect.\ \ref{SS:DATAASSUMPTIONS},
	the bootstrap assumptions of Subsect.\ \ref{SS:BOOTSTRAP},
	and the smallness assumptions of Subsect.\ \ref{SS:SMALLNESSASSUMPTIONS},
	the following pointwise estimates hold for
	$(t,\underline{x}) \in [0,T_{(Boot)}) \times \mathbb{R}^3$:
	\begin{align} \label{E:INHOMOGENEOUSTERMPOINTWISEBOUND}
	\left|
		F^{(k)}
	\right|
	& \leq C 
	\varepsilon |\nabla^{[2,k+1]} \Psi|,
	&& (P=1),
		\\
	\left|
		F^{(k)}
	\right|
	& \leq C \varepsilon (1 + \Psi) |\nabla^{k+1} \Psi|
	+
	\varepsilon |\nabla^{[2,k]} \Psi|,
	&& (P=2).
	\label{E:PISTWOINHOMOGENEOUSTERMPOINTWISEBOUND}
\end{align}
\end{lemma}

\begin{proof}
	We first consider the case $P=1$.
	Commuting \eqref{E:MODELWAVE} with $\nabla^k$,
	we compute that
	\[
	\left|
		F^{(k)}
	\right|
	\leq C
	\mathop{\sum_{a + b \leq k+2}}_{1 \leq a, \, 2 \leq b \leq k+1}
	\left| 
		\nabla^a \Psi
	\right|
	\left|
		\nabla^b \Psi
	\right|.
	\]
	The desired estimate \eqref{E:INHOMOGENEOUSTERMPOINTWISEBOUND}
	then follows as a simple consequence of 
	this bound and the bootstrap assumptions.
	The proof of \eqref{E:PISTWOINHOMOGENEOUSTERMPOINTWISEBOUND} is similar,
	the difference being that when $P=2$,
	we have the bound
	\[
	\left|
		F^{(k)}
	\right|
	\leq C \varepsilon (1 + \Psi) |\nabla^{k+1} \Psi|
	+
	C
	\mathop{\sum_{a + b \leq k+2}}_{1 \leq a \leq k, \, 2 \leq b \leq k}
	\left|
		\nabla^a \Psi
	\right|
	\left|
		\nabla^b \Psi
	\right|.
	\]
\end{proof}

\subsection{Energy estimates}
\label{SS:ENERGYESTIMATES}

We will use the following energy, which corresponds to between two and five commutations of the wave equation with $\nabla$, 
in order to control solutions.

\begin{definition}[\textbf{The energy}]
	\label{D:ENERGYDEF}
	We define
	\begin{align} \label{E:ENERGYDEF}
		\mathcal{E}_{[2,5]}(t)
		& :=
			\sum_{k'=2}^5
			\int_{\Sigma_t}
			|\partial_t \nabla^{k'} \Psi|^2
			+
			(1 + \Psi)^P |\nabla \nabla^{k'} \Psi|^2
			+ 
			|\nabla^{k'} \Psi|^2
		\, d \underline{x}.
	\end{align}
\end{definition}

We now provide the basic energy identity satisfied by solutions.

\begin{lemma}[\textbf{Basic energy identity}]
	\label{L:BASICENERGYIDENTITY}
Let $\Psi$ be a solution to the wave equation \eqref{E:MODELWAVE}.
Let $\mathcal{E}_{[2,5]}$ be the energy defined in \eqref{E:ENERGYDEF}	
and let $F^{(k)}$ be the inhomogeneous term
from \eqref{E:NABLAKOMMUTEDWAVE}.
Then for $P=1,2$, 
we have the energy identity
\begin{align} \label{E:BASICENERGYIDENTITY}
		\mathcal{E}_{[2,5]}(t)
		& 
		=
		\mathcal{E}_{[2,5]}(0)
		+ 
		P
		\sum_{k'=2}^5
		\int_{s=0}^t
		\int_{\Sigma_s}
			(\partial_t \Psi)
			(1 + \Psi)^{P-1}
			|\nabla \nabla^{k'} \Psi|^2
		\, d \underline{x}
		\, ds
			\\
	& \ \
		-
		2 P
		\sum_{k'=2}^5
		\int_{s=0}^t
		\int_{\Sigma_s}
			(1 + \Psi)^{P-1}
			(\nabla \Psi)
			\cdot
			(\nabla \nabla^{k'} \Psi)
			(\partial_t \nabla^{k'} \Psi)
		\, d \underline{x}
		\, ds
			\notag \\
		&
		\ \
		-
		2
		\sum_{k'=2}^5
		\int_{s=0}^t
		\int_{\Sigma_s}
			(\partial_t \nabla^{k'} \Psi)
			F^{(k')}
		\, d \underline{x}
		\, ds
		+
		2
		\sum_{k'=2}^5
		\int_{s=0}^t
		\int_{\Sigma_s}
			(\partial_t \nabla^{k'} \Psi)
			(\nabla^{k'} \Psi)
		\, d \underline{x}
		\, ds.
		\notag
		\end{align}
\end{lemma}

\begin{proof}
	The identity \eqref{E:BASICENERGYIDENTITY} 
	is standard and can verified by
	taking the time derivative of both sides of \eqref{E:ENERGYDEF},
	using equation \eqref{E:NABLAKOMMUTEDWAVE}
	for substitution, integrating by parts over $\Sigma_t$,
	and then integrating the resulting 
	identity in time.
\end{proof}

With the help of Lemma~\ref{L:BASICENERGYIDENTITY},
we now derive an inequality satisfied by the energy
$\mathcal{E}_{[2,5]}$.

\begin{proposition}[\textbf{Integral inequality for the energy}]
	\label{P:MAININTEGRALINEQUALITYFORENEGY}
	Let $\mathcal{E}_{[2,5]}$ be the energy defined in \eqref{E:ENERGYDEF}.
	Let
	$\mathbf{1}_{\lbrace -1 < \Psi \leq - \frac{1}{2} \rbrace}$
	be the characteristic function of the spacetime subset
	$\lbrace (t,\underline{x}) \ | \ -1 < \Psi(t,\underline{x}) \leq - \frac{1}{2} \rbrace$
  and define $\mathbf{1}_{\lbrace - \frac{1}{2} < \Psi  \rbrace}$ analogously.
	Let $\mathring{\updelta}_*$ be the data-size parameter from Def.\ \ref{D:CRUCIALDATASIZEPARAMETER}.
	Under the data-size assumptions of Subsect.\ \ref{SS:DATAASSUMPTIONS},
	the bootstrap assumptions of Subsect.\ \ref{SS:BOOTSTRAP},
	and the smallness assumptions of Subsect.\ \ref{SS:SMALLNESSASSUMPTIONS},
	the following integral inequality holds for
	$P=1,2$ and $t \in [0,T_{(Boot)})$:
	\begin{align} \label{E:MAININTEGRALINEQUALITYFORENEGY}
		&
		\mathcal{E}_{[2,5]}(t)
		+
		\frac{P}{8}
		\mathring{\updelta}_*
		\sum_{k'=2}^5
		\int_{s=0}^t
		\int_{\Sigma_s}
			\mathbf{1}_{\lbrace -1 < \Psi \leq - \frac{1}{2} \rbrace}
			(1 + \Psi)^{P-1}
			|\nabla \nabla^{k'} \Psi|^2
		\, d \underline{x}
		\, ds
			\\
		& \leq
		\mathcal{E}_{[2,5]}(0)
		+ 
		C
		\sum_{k'=2}^5
		\int_{s=0}^t
			\int_{\Sigma_s}
				|\partial_t \nabla^{k'} \Psi|^2
			\, d \underline{x}
		\, ds
		+ 
		C
		\sum_{k'=2}^5
		\int_{s=0}^t
			\int_{\Sigma_s}
				|\nabla^{k'} \Psi|^2
			\, d \underline{x}
		\, ds
		\notag	\\
	& \ \
		+
		C
		\sum_{k'=2}^5
		\int_{s=0}^t
		\int_{\Sigma_s}
			\mathbf{1}_{\lbrace - \frac{1}{2} < \Psi  \rbrace}
			(1 + \Psi)^{P-1}
			|\nabla \nabla^{k'} \Psi|^2
		\, d \underline{x}
		\, ds
			\notag \\
	& \ \
		+
		C \varepsilon 
		\sum_{k'=2}^5
		\int_{s=0}^t
		\int_{\Sigma_s}
			\mathbf{1}_{\lbrace -1 < \Psi \leq - \frac{1}{2} \rbrace}
			(1 + \Psi)^{2(P-1)}
			|\nabla \nabla^{k'} \Psi|^2
		\, d \underline{x}
		\, ds.
		\notag
		\end{align}
\end{proposition}

\begin{proof}
	We must bound the terms appearing in the energy identity \eqref{E:BASICENERGYIDENTITY}.
	We give the proof only for the case $P=1$ since the case $P=2$ can be handled using
	similar arguments.
	We start by bounding the first sum on the right-hand side of \eqref{E:BASICENERGYIDENTITY};
	this is the only term that requires careful treatment.
	We split the integration domain $[0,t] \times \mathbb{R}^3$
	into two pieces: a piece in which $-1 < \Psi \leq - \frac{1}{2}$
	and a piece in which $\Psi > - \frac{1}{2}$.
	To bound the first piece, we use the estimate 
	\eqref{E:PSISMALLIMPLIESPARTIALTPSINEGATIVE}
	to deduce that whenever $-1 < \Psi \leq - \frac{1}{2}$,
	the integrand satisfies
	$
	\displaystyle
	(\partial_t \Psi) |\nabla \nabla^{k'} \Psi|^2
	\leq
	-
		\frac{1}{8}
		\mathring{\updelta}_*
		|\nabla \nabla^{k'} \Psi|^2
	$. 
	We can therefore bring all of the corresponding integrals
	over to the left-hand side of \eqref{E:MAININTEGRALINEQUALITYFORENEGY} as 
	\emph{positive} integrals, as is indicated there.
	To bound the second piece, we use
	the estimate \eqref{E:PARTIALTPSIWELLAPPROXIMATED}
	to bound $\partial_t \Psi$ in $L^{\infty}$ by $\leq C$,
	which allows us to bound the integrand by
	$
	\leq C
	|\nabla \nabla^{k'} \Psi|^2
	$.	
	It follows that since $\Psi > - \frac{1}{2}$ (by assumption),
	the integrals under consideration are bounded by
	the third sum on the right-hand side of \eqref{E:MAININTEGRALINEQUALITYFORENEGY}.
	
	To bound the second sum on the right-hand side of \eqref{E:BASICENERGYIDENTITY},
	we first use the bootstrap assumption \eqref{E:SMALLLINFTYBOOTSTRAP}
	to bound the integrand factor $\nabla \Psi$ in $L^{\infty}$ by
	$\leq \varepsilon$. Thus, using Young's inequality,
	we bound the terms under consideration by $\leq$
	the sum of the first, third, and fourth sums on the right-hand side of \eqref{E:MAININTEGRALINEQUALITYFORENEGY}.
	
	To bound the third sum on the right-hand side of \eqref{E:BASICENERGYIDENTITY},
	we use \eqref{E:INHOMOGENEOUSTERMPOINTWISEBOUND} 
	and Young's inequality
	to bound the integrand by
	$
	\leq
	C
	\varepsilon 
	\sum_{k'=2}^5
	|\partial_t \nabla^{k'} \Psi|^2
	+
	C
	\varepsilon 
	\sum_{k'=2}^6
	|\nabla^{k'} \Psi|^2
	$.
	It is easy to see that the corresponding integrals
	are bounded by the right-hand side of \eqref{E:MAININTEGRALINEQUALITYFORENEGY}.
	
	Finally, using Young's inequality, we bound last sum on the right-hand side of \eqref{E:BASICENERGYIDENTITY}
	by the first two sums on the right-hand side of \eqref{E:MAININTEGRALINEQUALITYFORENEGY}.
\end{proof}

In the next proposition, 
we use Prop.\ \ref{P:MAININTEGRALINEQUALITYFORENEGY}
to derive our main a priori energy estimates. 
We also derive improvements of the bootstrap assumptions
\eqref{E:PSIITSELFBOOTSTRAP}-\eqref{E:SMALLLINFTYBOOTSTRAP}.

\begin{proposition}[\textbf{A priori energy estimates and improvement of the bootstrap assumptions}]
	\label{P:APRIORIESTIMATES}
	Let $\mathring{\updelta}_*$
	be the data-size parameter from
	\eqref{E:CRUCIALDATASIZEPARAMETER}
	and let $\mathbf{1}_{\lbrace -1 < \Psi \leq - \frac{1}{2} \rbrace}$
	be the characteristic function of the spacetime subset
	$\lbrace (t,\underline{x}) \ | \ -1 < \Psi(t,\underline{x}) \leq - \frac{1}{2} \rbrace$.
	There exists a constant $C > 0$ such that
	under the data-size assumptions of Subsect.\ \ref{SS:DATAASSUMPTIONS},
	the bootstrap assumptions of Subsect.\ \ref{SS:BOOTSTRAP},
	and the smallness assumptions of Subsect.\ \ref{SS:SMALLNESSASSUMPTIONS},
	the following a priori energy estimate holds for
	$P=1,2$ and $t \in [0,T_{(Boot)})$:
	\begin{align} \label{E:MAINAPRIORIENERGYESTIMATE}
		\mathcal{E}_{[2,5]}(t)
		+
		\frac{P}{16}
		\mathring{\updelta}_*
		\sum_{k'=2}^5
		\int_{s=0}^t
		\int_{\Sigma_s}
			\mathbf{1}_{\lbrace -1 < \Psi \leq - \frac{1}{2} \rbrace}
			(1 + \Psi)^{P-1}
			|\nabla \nabla^{k'} \Psi|^2
		\, d \underline{x}
		\, ds
		& \leq C \mathring{\upepsilon}^2.
	\end{align}
	Moreover, we have the following estimates,
	which are a strict improvement of the bootstrap
	assumptions 
	\eqref{E:PSIITSELFBOOTSTRAP}-\eqref{E:SMALLLINFTYBOOTSTRAP}
	for $\mathring{\upepsilon}$
	sufficiently small:
	\begin{subequations}
	\begin{align} \label{E:PSIITSELFBOOTSTRAPIMPROVED}
	\| \Psi \|_{L^{\infty}(\Sigma_t)}
	& \leq 2 \mathring{\updelta}_*^{-1} \mathring{\updelta}	
		+ C \mathring{\upepsilon},
			\\
	\| \partial_t \Psi \|_{L^{\infty}(\Sigma_t)}
	&  \leq 
		\mathring{\updelta}	
		+ C \mathring{\upepsilon},
			\label{E:PARTIALTPSIBOOTSTRAPIMPROVED}
				\\
	\label{E:BOOTSTRAPIMPROVED}
		\| \nabla^{[1,3]} \Psi \|_{L^{\infty}(\Sigma_t)},
			\,
		\| \partial_t \nabla^{[1,3]} \Psi \|_{L^{\infty}(\Sigma_t)},
			\,
		\| \partial_t^2 \nabla^{\leq 1} \Psi \|_{L^{\infty}(\Sigma_t)}
	& \leq C \mathring{\upepsilon}.
\end{align}
\end{subequations}
\end{proposition}

\begin{proof}
	We give the proof only for the case $P=1$ since the case $P=2$ can be handled using
	similar arguments. To obtain \eqref{E:MAINAPRIORIENERGYESTIMATE},
	we first note that for $\varepsilon$ sufficiently small
	relative to $\mathring{\updelta}_*$, we can absorb the last sum 
	on the right-hand side of \eqref{E:MAININTEGRALINEQUALITYFORENEGY}
	into the second term on the left-hand side, 
	which at most reduces its coefficient from 
	$\frac{1}{8} \mathring{\updelta}_*$
	to 
	$\frac{1}{16} \mathring{\updelta}_*$,
	as is stated on the left-hand side of \eqref{E:MAINAPRIORIENERGYESTIMATE}.
	Moreover, 
	since $\mathbf{1}_{\lbrace - \frac{1}{2} < \Psi  \rbrace} \leq C \mathbf{1}_{\lbrace - \frac{1}{2} < \Psi  \rbrace} (1 + \Psi)$,
	we have the bound
	$
	\int_{\Sigma_s}
			\mathbf{1}_{\lbrace - \frac{1}{2} < \Psi  \rbrace}
			|\nabla \nabla^{k'} \Psi|^2
		\, d \underline{x}
	\leq C \mathcal{E}_{[2,5]}(s)
	$
	for the terms in the next-to-last sum
	on the right-hand side of \eqref{E:MAININTEGRALINEQUALITYFORENEGY}.
	The remaining $\Sigma_s$ integrals are easily seen to be
	bounded in magnitude by
	$\leq C \mathcal{E}_{[2,5]}(s)$. 
	Also using the data bound
	$
	\mathcal{E}_{[2,5]}(0)
	\leq C \mathring{\upepsilon}^2
	$,
	which follows from our data assumptions \eqref{E:DATASIZE},
	we obtain
	\begin{align} \label{E:GRONWALLREADYENERGYESTIMATE}
	\mathcal{E}_{[2,5]}(t)
	& +
		\frac{1}{16}
		\mathring{\updelta}_*
		\sum_{k'=2}^5
		\int_{s=0}^t
		\int_{\Sigma_s}
			\mathbf{1}_{\lbrace -1 < \Psi \leq - \frac{1}{2} \rbrace}
			|\nabla \nabla^{k'} \Psi|^2
		\, d \underline{x}
		\, ds
			\\
		& \leq
		C \mathring{\upepsilon}^2
		+
		c
		\int_{s=0}^t
			\mathcal{E}_{[2,5]}(s)
		\, ds.
		\notag
	\end{align}
	From \eqref{E:GRONWALLREADYENERGYESTIMATE}, Gr\"{o}nwall's inequality,
	and \eqref{E:BOOTSTRAPTIME},
	we conclude
	\[
	\mathcal{E}_{[2,5]}(t)
		+
		\frac{1}{16}
		\mathring{\updelta}_*
		\sum_{k'=2}^5
		\int_{s=0}^t
		\int_{\Sigma_s}
			\mathbf{1}_{\lbrace -1 < \Psi \leq - \frac{1}{2} \rbrace}
			|\nabla \nabla^{k'} \Psi|^2
		\, d \underline{x}
		\, ds
		\leq C \exp(ct) \mathring{\upepsilon}^2
		\leq C \mathring{\upepsilon}^2,
	\]
	which is the desired bound \eqref{E:MAINAPRIORIENERGYESTIMATE}.
	
	The estimates \eqref{E:BOOTSTRAPIMPROVED} for 
	$\nabla^{[2,3]} \Psi$
	and
	$\partial_t \nabla^{[2,3]} \Psi$
	then follow from \eqref{E:MAINAPRIORIENERGYESTIMATE}
	and the Sobolev embedding result
	$
	H^2(\mathbb{R}^3)
	\hookrightarrow
	L^{\infty}(\mathbb{R}^3)
	$.
	Next, we take up to one $\nabla$ derivative
	of equation \eqref{E:MODELWAVE}
	and use the already obtained 
	$L^{\infty}$ estimates and the bootstrap assumptions
	to obtain the bound
	$
	\| \partial_t^2 \nabla^{\leq 1} \Psi \|_{L^{\infty}(\Sigma_t)}
	\leq 
	C \mathring{\upepsilon}
	$
  stated in \eqref{E:BOOTSTRAPIMPROVED}.
	Using this bound, the fundamental theorem of calculus,
	and the data assumptions
	$
	\| \nabla \mathring{\Psi}_0 \|_{L^{\infty}(\Sigma_0)}
	\leq \mathring{\upepsilon}
	$
	and
	$
	\| \mathring{\Psi}_0 \|_{L^{\infty}(\Sigma_0)}
	\leq \mathring{\updelta}
	$,
	we obtain the bounds
	$
	\| \partial_t \nabla \Psi \|_{L^{\infty}(\Sigma_t)}
	\leq 
	C \mathring{\upepsilon}
	$
	and 
	$
	\| \partial_t \Psi \|_{L^{\infty}(\Sigma_t)}
	\leq 
	\mathring{\updelta}
	+
	C \mathring{\upepsilon}
	$,
	which in particular yields \eqref{E:PARTIALTPSIBOOTSTRAPIMPROVED}.
	Using a similar argument based on the already obtained bound
	$\| \partial_t \nabla \Psi \|_{L^{\infty}(\Sigma_t)} \leq C \mathring{\upepsilon}$,
	we deduce
	$
	\| \nabla \Psi \|_{L^{\infty}(\Sigma_t)}
	\leq 
	C \mathring{\upepsilon}
	$.
	Similarly, from the
	already obtained bound
	$
	\| \partial_t \Psi \|_{L^{\infty}(\Sigma_t)}
	\leq 
	\mathring{\updelta}
	+
	C \mathring{\upepsilon}
	$,
	the fundamental theorem of calculus,
	the initial data bound
	$
	\| \mathring{\Psi}  \|_{L^{\infty}(\Sigma_0)}
	\leq 
	\mathring{\upepsilon}
	$,
	and the fact that $0 \leq t < T_{(Boot)} \leq 2 \mathring{\updelta}_*^{-1}$,
	we deduce 
	$
	\| \Psi \|_{L^{\infty}(\Sigma_t)}
	\leq 2 \mathring{\updelta}_*^{-1} \mathring{\updelta}	
	+ C \mathring{\upepsilon}
	$,
	that is, \eqref{E:PSIITSELFBOOTSTRAPIMPROVED}.
	
\end{proof}

\section{The main results}
\label{S:MAINRESULTS}
We now derive our main results, namely
Theorem~\ref{T:STABLEFINITETIMEBREAKDOWN}
and Prop.\ \ref{P:BLOWUPOFKRETSCHMANN}.

\begin{theorem}[\textbf{Stable finite-time degeneration of hyperbolicity}]
\label{T:STABLEFINITETIMEBREAKDOWN}
	Let $(\mathring{\Psi},\mathring{\Psi}_0) \in H^6(\mathbb{R}^3) \times H^5(\mathbb{R}^3)$
	be compactly supported initial data \eqref{E:MODELDATA} for the wave equation \eqref{E:MODELWAVE}
	with $P \in \lbrace 1,2 \rbrace$ and let $\Psi$ denote the corresponding solution.
	Let
	\begin{align} \label{E:ONEPLUSPSIMINSIGMATDEF}
		\mathcal{M}(t)
		& := \min_{(s,\underline{x}) \in [0,t] \times \mathbb{R}^3} 
		\left\lbrace 1 + \Psi(s,\underline{x}) \right\rbrace.
	\end{align}
	Let $\mathring{\upepsilon}$, 
	$\mathring{\updelta}$,
	and $\mathring{\updelta}_*$
	be the data-size parameters 
	from \eqref{E:DATASIZE}-\eqref{E:CRUCIALDATASIZEPARAMETER}
	and assume that
	$\mathring{\updelta} > 0$
	and $\mathring{\updelta}_* > 0$.
	Note that
	if $\mathring{\upepsilon}$ is sufficiently small,
	then $\mathcal{M}(0) = 1 + \mathcal{O}(\mathring{\upepsilon}) > 0$.
	If $\mathring{\upepsilon}$
	is sufficiently small relative to
	$\mathring{\updelta}^{-1}$
	and
	$\mathring{\updelta}_*$
	in the sense described in Subsect.\ \ref{SS:SMALLNESSASSUMPTIONS},
	then the following conclusions hold.
	
	\medskip
	
	\noindent \underline{\textbf{Breakdown in hyperbolicity precisely at time} $T_{\star}$}:
	There exists a $T_{\star} > 0$ satisfying
	\begin{align} \label{E:HYPERBOLICITYBREAKDOWNTIME}
		T_{\star}
		= \left\lbrace
				1 + \mathcal{O}(\mathring{\upepsilon})
			\right\rbrace
			\mathring{\updelta}_*^{-1}
	\end{align}
	such that the solution exists classically on the slab 
	$[0,T_{\star}) \times \mathbb{R}^3$
	and such that the following inequality holds for $0 \leq t < T_{\star}$:
	\begin{align} \label{E:POSITIVEMIN}
		\mathcal{M}(t) > 0.
	\end{align}
	Moreover,
	\begin{align} \label{E:HYPERBOLICITYBREAKDSDOWN}
		\lim_{t \uparrow T_{\star}} \mathcal{M}(t) = 0.
	\end{align}
	
	\medskip
	\noindent \underline{\textbf{Regularity properties on} $[0,T_{\star}) \times \mathbb{R}^3$}:
	On the slab
	$[0,T_{\star}) \times \mathbb{R}^3$,
	the solution satisfies the energy bounds
	\eqref{E:MAINAPRIORIENERGYESTIMATE},
	the $L^{\infty}$ estimates \eqref{E:PSIITSELFBOOTSTRAPIMPROVED}-\eqref{E:BOOTSTRAPIMPROVED},
	and the pointwise estimates
	\eqref{E:PSIWELLAPPROXIMATED}-\eqref{E:PARTIALTPSIWELLAPPROXIMATED}
	(with $C \mathring{\upepsilon}$ on the right-hand side in place of $\varepsilon$
	in the latter two estimates).
	Moreover,
	\begin{subequations}
	\begin{align}
		\Psi
		& \in 
		C\left([0,T_{\star}),H^6(\mathbb{R}^3) \right)
		\cap
		L^{\infty}\left([0,T_{\star}),H^5(\mathbb{R}^3) \right),
			\label{E:PSIREGULARITYBEFOREDEGENERACY} \\
		\partial_t \Psi
		& \in 
		C\left([0,T_{\star}),H^5(\mathbb{R}^3) \right)
		\cap
		L^{\infty}\left([0,T_{\star}),H^5(\mathbb{R}^3) \right).
		\label{E:PARTIALTPSIREGULARITYBEFOREDEGENERACY}
	\end{align}
	\end{subequations}
	
	\medskip
	\noindent \underline{\textbf{Regularity properties on} $[0,T_{\star}] \times \mathbb{R}^3$}:
	$\Psi$ extends to a classical solution on 
	the closed slab $[0,T_{\star}] \times \mathbb{R}^3$
	enjoying the following regularity properties:
	for any $N < 5$, we have
	\begin{subequations}
	\begin{align}
		\Psi
		& \in 
			C\left([0,T_{\star}],H^5(\mathbb{R}^3) \right),
			\label{E:PSIREGULARITYAFTERDEGENERACY} \\
		\partial_t \Psi
		& \in 
		C\left([0,T_{\star}],H^N(\mathbb{R}^3) \right)
		\cap
		L^{\infty}\left([0,T_{\star}],H^5(\mathbb{R}^3) \right).
		\label{E:PARTIALTPSIREGULARITYAFTERDEGENERACY}
	\end{align}
	\end{subequations}
	In particular, 
	the $L^{\infty}$ estimates \eqref{E:PSIITSELFBOOTSTRAPIMPROVED}-\eqref{E:BOOTSTRAPIMPROVED}
	and the pointwise estimates
	\eqref{E:PSIWELLAPPROXIMATED}-\eqref{E:PARTIALTPSIWELLAPPROXIMATED}
	(with $C \mathring{\upepsilon}$ on the right-hand side in place of $\varepsilon$ in these estimates)
	hold on $[0,T_{\star}] \times \mathbb{R}^3$. Moreover, in the case\footnote{In the proof of the theorem,
	we clarify why our proof of \eqref{E:L2INTEGRABILITYATTOPORDER} relies on the assumption $P=1$.} $P=1$,
	we have
	\begin{align} \label{E:L2INTEGRABILITYATTOPORDER}
		\Psi
		& \in 
			L^2\left([0,T_{\star}],H^6(\mathbb{R}^3) \right).
	\end{align}

	\medskip
	\noindent \underline{\textbf{Description of the breakdown along} $\Sigma_{T_{\star}}$}:
	The set
	\begin{align} \label{E:BREAKDOWNINHYPERBOLICITYSET}
		\Sigma_{T_{\star}}^{Degen}
		& := 
		\lbrace
			q \in \Sigma_{T_{\star}} \ | \ 1 + \Psi(q) = 0
		\rbrace
	\end{align}
	is non-empty and we have the estimate
	\begin{align} \label{E:PARTIALTPSINEGATIVEINSIGMATDEGENERACY}
	\sup_{\Sigma_{T_{\star}}^{Degeneracy}} \partial_t \Psi
	\leq 
		-
		\frac{1}{8}
		\mathring{\updelta}_*.
	\end{align}
	
	In particular, in the case $P=1$,
	the hyperbolicity of the wave equation breaks down
	on $\Sigma_{T_{\star}}^{Degen}$ in the following sense:
	if $q \in \Sigma_{T_{\star}}^{Degen}$,
	then any $C^1$ extension of $\Psi$
	to any spacetime neighborhood
	$\Omega_q$ containing $q$ necessarily contains points $p$ 
	such that equation \eqref{E:MODELWAVE} is elliptic at $\Psi(p)$.
	In contrast, in the case $P=2$, only the \underline{strict} hyperbolicity (in the sense of Footnote~\ref{FN:STRICTLYHYPERBOLIC})
	of equation \eqref{E:MODELWAVE} breaks down for the first time
	at $T_{\star}$.
	
\end{theorem}

\begin{proof}
	Let $T_{\star}$ be the supremum of the set of times $T_{(Boot)}$
	subject to inequality \eqref{E:BOOTSTRAPTIME}
	and such that the solution exists classically on
	the slab $[0,T_{(Boot)}) \times \mathbb{R}^{3}$,
	has the same Sobolev regularity as the initial data,
	and satisfies the bootstrap assumptions of
	Subsect.\ \ref{SS:BOOTSTRAP}
	with $\varepsilon := C_* \mathring{\upepsilon}$,
	where $C_*$ is described just below.
	By standard local well-posedness,
	see for example, \cite{lH1997},
	if $\mathring{\upepsilon}$ is
	sufficiently small 
	and $C_* > 1$ is sufficiently large, 
	note that this is consistent with the assumed inequalities \eqref{E:DATAEPSILONVSBOOTSTRAPEPSILON},
	then $T_{\star} > 0$.
	Next, we state the following standard continuation result, 
	which can be proved, for example, by making straightforward modifications
	to the proof of \cite{lH1997}*{Theorem 6.4.11}:
	if $T_{\star} < 2 \mathring{\updelta}_*^{-1}$, then the solution can be classically continued to a slab of the form
	$[0,T_{\star} + \triangle] \times \mathbb{R}^{3}$ (for some $\triangle > 0$ with $T_{\star} + \triangle < 2 \mathring{\updelta}_*^{-1}$)
	on which the solution has the same Sobolev regularity as the initial data and on which the bootstrap assumptions hold,
	as long as 
	the bootstrap inequalities \eqref{E:PSIITSELFBOOTSTRAP}-\eqref{E:SMALLLINFTYBOOTSTRAP}
	are strictly satisfied for $t \in [0,T_{\star})$
	and
	$
	\inf_{t \in [0,T_{\star})} \mathcal{M}(t) 
	> 0
	$. 	
	It follows that either \textbf{i)} $T_{\star} = 2 \mathring{\updelta}_*^{-1}$,
	\textbf{ii)} that the bootstrap inequalities \eqref{E:PSIITSELFBOOTSTRAP}-\eqref{E:SMALLLINFTYBOOTSTRAP}
	are saturated at some time $t \in [0,T_{\star})$,
	or \textbf{iii)} 
	that
	$
	\inf_{t \in [0,T_{\star})} \mathcal{M}(t) 
	= 0
	$.
	If $C_*$ is chosen to be sufficiently large and $\mathring{\upepsilon}$ is chosen to be sufficiently small, 
	then the a priori estimates
	\eqref{E:PSIITSELFBOOTSTRAPIMPROVED}-\eqref{E:BOOTSTRAPIMPROVED}
	ensure that
	the bootstrap inequalities \eqref{E:PSIITSELFBOOTSTRAP}-\eqref{E:SMALLLINFTYBOOTSTRAP}
	are in fact strictly satisfied for $t \in [0,T_{\star})$.
	Moreover, from definition \eqref{E:CRUCIALDATASIZEPARAMETER}
	and the estimate \eqref{E:PSIWELLAPPROXIMATED}
	(which is now known to hold with $\varepsilon$ replaced by $C \mathring{\upepsilon}$),
	we see that
	$\min_{\Sigma_t} (1 + \Psi) 
	= 1 - \mathring{\updelta}_* t + \mathcal{O}(\mathring{\upepsilon})$
	and thus, in fact, case \textbf{iii)} occurs with
	$T_{\star} 
	=
	\mathring{\updelta}_*^{-1}
	+
	\mathcal{O}(\mathring{\upepsilon})
	= 
	\left\lbrace
		1 + \mathcal{O}(\mathring{\upepsilon})
	\right\rbrace
	\mathring{\updelta}_*^{-1}
	< 2 \mathring{\updelta}_*^{-1}
	$
	and
	$\lim_{t \uparrow T_{\star}} \mathcal{M}(t) = 0$.
	From the above reasoning, we easily deduce that 
	the energy bound \eqref{E:MAINAPRIORIENERGYESTIMATE} holds
	for $t \in [0,T_{\star})$
	and, since the a priori estimates
	\eqref{E:PSIITSELFBOOTSTRAPIMPROVED}-\eqref{E:BOOTSTRAPIMPROVED}
	show that the bootstrap assumptions hold with 
	$\varepsilon$ replaced by $C \mathring{\upepsilon}$,
	that the pointwise estimates
	\eqref{E:PSIWELLAPPROXIMATED}-\eqref{E:PARTIALTPSIWELLAPPROXIMATED}
	hold for $(t,\underline{x}) \in [0,T_{\star}) \times \mathbb{R}^{3}$
	with $\varepsilon$ replaced by $C \mathring{\upepsilon}$.
	
	In the rest of this proof,
	we sometimes silently use the simple facts that
	$\Psi, \partial_t \Psi \in L^{\infty}\left([0,T_{\star}),H^1(\mathbb{R}^3)\right)$.
	These facts do not follow from 
	the energy estimates \eqref{E:MAINAPRIORIENERGYESTIMATE}
	(since the energy does not directly control $\Psi$, $\partial_t \Psi$, $\nabla \Psi$, or $\nabla \partial_t \Psi$),
	but instead follow from
	\eqref{E:PSIITSELFBOOTSTRAPIMPROVED}-\eqref{E:BOOTSTRAPIMPROVED}
	and the compactly supported (in space) nature of the solution.
	To proceed, we easily conclude from the definition of $\mathcal{E}_{[2,5]}(t)$
	and the fact that the estimates \eqref{E:MAINAPRIORIENERGYESTIMATE} and \eqref{E:PARTIALTPSIBOOTSTRAPIMPROVED}-\eqref{E:BOOTSTRAPIMPROVED}
	hold on $[0,T_{\star})$ that
	$
	\partial_t \Psi
	\in 
	L^{\infty}\left([0,T_{\star}],H^5(\mathbb{R}^3) \right),
	$
	as is stated in \eqref{E:PARTIALTPSIREGULARITYAFTERDEGENERACY}.
	Also, this fact trivially implies the corresponding statement
	in \eqref{E:PARTIALTPSIREGULARITYBEFOREDEGENERACY},
	where the closed time interval is replaced with $[0,T_{\star})$.
	The facts that
	$\Psi
		\in 
		C\left([0,T_{\star}),H^6(\mathbb{R}^3) \right)
	$
	and 
	$
	\partial_t \Psi
	\in 
		C\left([0,T_{\star}),H^5(\mathbb{R}^3) \right)
	$,
	as is stated in \eqref{E:PSIREGULARITYBEFOREDEGENERACY} and \eqref{E:PARTIALTPSIREGULARITYBEFOREDEGENERACY},
	are standard results that can be proved using 
	energy estimates and simple facts from functional analysis.
	We omit the details and instead refer the reader to 
	\cite{jS2008c}*{Section 2.7.5}. We note that in proving these ``soft'' facts,
	it is important that $\mathcal{M}(t) > 0$ for $t \in [0,T_{\star})$,
	which implies that standard techniques for
	strictly hyperbolic equations can be used.
	To obtain the conclusion 
	$\Psi
		\in 
		L^2\left([0,T_{\star}],H^6(\mathbb{R}^3) \right)
		$
	in the case $P=1$,
	as is stated in \eqref{E:L2INTEGRABILITYATTOPORDER},
	we simply use the fact that the energy bounds \eqref{E:MAINAPRIORIENERGYESTIMATE} 
	hold on $[0,T_{\star})$
	(including the bound for the spacetime integral term on the left-hand side). 
	Note that the same argument does not apply in the case $P=2$ since in this case, the spacetime integral
	on the left-hand side of \eqref{E:MAINAPRIORIENERGYESTIMATE} features the degenerate weight $1 + \Psi$.
	The fact that
	$\Psi
		\in 
		C\left([0,T_{\star}],H^5(\mathbb{R}^3) \right)
	$,
	as is stated in \eqref{E:PSIREGULARITYAFTERDEGENERACY},
	is a simple consequence of the fundamental theorem of calculus,
	the fact that $\mathring{\Psi} \in H^6(\Sigma_0)$,
	and the already proven fact that
	$\partial_t \Psi
		\in 
		L^{\infty}\left([0,T_{\star}],H^5(\mathbb{R}^3) \right)
	$.
	To obtain that for $N < 5$, we have
	$\partial_t \Psi
		\in 
			C\left([0,T_{\star}],H^N(\mathbb{R}^3) \right)
	$,
	as is stated in \eqref{E:PARTIALTPSIREGULARITYAFTERDEGENERACY},
	we first use equation \eqref{E:MODELWAVE},
	the	fact that
	$
	\Psi
		\in 
		C\left([0,T_{\star}],H^5(\mathbb{R}^3) \right)
	$,
	and the standard Sobolev calculus
	to obtain
	$\partial_t^2 \Psi
		\in 
			C\left([0,T_{\star}],H^3(\mathbb{R}^3) \right)
	$.
	Hence, 
	from the fundamental theorem of calculus
	and the fact that $\mathring{\Psi}_0 \in H^5(\Sigma_0)$,
	we obtain
	$\partial_t \Psi
		\in 
			C\left([0,T_{\star}],H^3(\mathbb{R}^3) \right)
	$.
	From this fact and the fact
	$\partial_t \Psi
		\in 
		L^{\infty}\left([0,T_{\star}],H^5(\mathbb{R}^3) \right)
	$,
	we obtain, by interpolating\footnote{Here, we mean the following standard inequality: 
	if $f \in H^5(\Sigma_t)$ and $0 \leq N \leq 5$, 
	then there exists a constant $C_N > 0$ such that
	$\| f \|_{H^N(\Sigma_t)} \leq C_N \| f \|_{L^2(\Sigma_t)}^{1-N/5} \| f \|_{H^5(\Sigma_t)}^{N/5}$. \label{FN:INTERPOLATION}} 
	between $L^2$ and $H^5$,
	the desired conclusion
	$
	\partial_t \Psi
		\in 
			C\left([0,T_{\star}],H^N(\mathbb{R}^3) \right)
	$.
	
	Next, we note that 
	the arguments given in the first paragraph of this proof imply that 
	$\mathcal{M}$ extends as a continuous decreasing function defined 
	for $t \in [0,T_{\star}]$ such that $\mathcal{M}(t) > 0$ for $t \in [0,T_{\star})$
	and such that $\mathcal{M}(T_{\star}) = 0$.
	Also using that
	$
	\Psi
		\in 
		C\left([0,T_{\star}],H^5(\mathbb{R}^3) \right)
		\subset
		C\left([0,T_{\star}],C^3(\mathbb{R}^3) \right)
	$,
	we deduce, in view of
	definitions \eqref{E:ONEPLUSPSIMINSIGMATDEF}
	and
	\eqref{E:BREAKDOWNINHYPERBOLICITYSET},
	that
	$\Sigma_{T_{\star}}^{Degen}$
	is non-empty.
	Moreover, from \eqref{E:PSISMALLIMPLIESPARTIALTPSINEGATIVE}
	and the fact that
	$
	\partial_t 
		\Psi
		\in 
		C\left([0,T_{\star}],H^{4.9}(\mathbb{R}^3) \right)
		\subset
		C\left([0,T_{\star}],C^3(\mathbb{R}^3) \right)
	$,
	we find that the estimate \eqref{E:PARTIALTPSINEGATIVEINSIGMATDEGENERACY} 
	holds.
	In addition, in view of \eqref{E:PARTIALTPSINEGATIVEINSIGMATDEGENERACY},
	we see that in the case $P=1$, 
	if $q \in \Sigma_{T_{\star}}^{Degen}$, 
	then any $C^1$ extension of $\Psi$ to a neighborhood
	of $q$ contains points $p$ such that $1 + \Psi(p) < 0$,
	which renders equation \eqref{E:MODELWAVE} elliptic.
	This is in contrast to the case $P=2$ 
	in the sense that equation \eqref{E:MODELWAVE} is hyperbolic
	for all values of $\Psi$.
	
	\end{proof}

Theorem~\ref{T:STABLEFINITETIMEBREAKDOWN} yields that $\Psi$ remains regular, 
all the way up to the time $T_{\star}$.
However, as the next proposition shows,
a type of invariant blowup does in fact occur at time $T_{\star}$ in both 
the cases $P=1,2$.
The blowup is tied to the Riemann curvature of the metric $g$.

\begin{proposition}[\textbf{Blowup of the Kretschmann scalar}]
	\label{P:BLOWUPOFKRETSCHMANN}
	Let $g = g(\Psi)$ denote the spacetime metric defined in \eqref{E:SPACETIMEMETRIC}
	and let $Riem(g)$ denote the Riemann curvature tensor\footnote{Our sign convention for 
	curvature is
	$
		D_{\alpha} D_{\beta} X_{\mu} 
		- 
		D_{\beta} D_{\alpha} X_{\mu} 	
		= Riem(g)_{\alpha \beta \mu \nu}X^{\nu}
	$,
	where $D$ denotes the Levi--Civita connection of $g$
	and $X$ is an arbitrary smooth vector field.
	} 
	of $g$.
	Under the assumptions and conclusions of Theorem~\ref{T:STABLEFINITETIMEBREAKDOWN},
	we have the following estimate for the Kretschmann scalar
	$Riem(g)^{\alpha \beta \gamma \delta} Riem(g)_{\alpha \beta \gamma \delta}$
	on $[0,T_{\star}] \times \mathbb{R}^3$:
	\begin{subequations}
	\begin{align} \label{E:BLOWUPOFKRETSCHMANN}
		Riem(g)^{\alpha \beta \gamma \delta} Riem(g)_{\alpha \beta \gamma \delta}
		& = \frac{15}{2} \frac{(\partial_t \Psi)^4}{(1 + \Psi)^4}
			+ \mathcal{O}
				\left(\frac{\mathring{\upepsilon}}{(1 + \Psi)^3}
				\right),
		&& (P=1),
				\\
	Riem(g)^{\alpha \beta \gamma \delta} Riem(g)_{\alpha \beta \gamma \delta}
	& = 60 \frac{(\partial_t \Psi)^4}{(1 + \Psi)^4}
			+ \mathcal{O}
				\left(\frac{\mathring{\upepsilon}}{(1 + \Psi)^3}
				\right),
	&& (P=2).
	\label{E:PISTWOBLOWUPOFKRETSCHMANN}
\end{align}
\end{subequations}
In particular, 
$Riem(g)^{\alpha \beta \gamma \delta} Riem(g)_{\alpha \beta \gamma \delta}$
is bounded for $0 \leq t < T_{\star}$,
while by \eqref{E:PSISMALLIMPLIESPARTIALTPSINEGATIVE}
and
\eqref{E:BLOWUPOFKRETSCHMANN}-\eqref{E:PISTWOBLOWUPOFKRETSCHMANN},
at time $T_{\star}$,
$Riem(g)^{\alpha \beta \gamma \delta} Riem(g)_{\alpha \beta \gamma \delta}$
blows up precisely on the subset
$\Sigma_{T_{\star}}^{Degen}$
defined in \eqref{E:BREAKDOWNINHYPERBOLICITYSET}.
\end{proposition}

\begin{proof}
We prove only \eqref{E:BLOWUPOFKRETSCHMANN} since the proof of \eqref{E:PISTWOBLOWUPOFKRETSCHMANN}
is similar. The identities that we state in this proof rely on the form of the metric \eqref{E:SPACETIMEMETRIC}.
We first note the following simple decomposition formula,
which relies on the standard symmetry and anti-symmetry properties
of the Riemann curvature tensor:
\begin{align}
Riem(g)^{\alpha \beta \gamma \delta} Riem(g)_{\alpha \beta \gamma \delta}
& = Riem(g)_{ab}^{\ \ cd} Riem(g)_{cd}^{\ \ ab}
	+ 4 Riem(g)_{a0}^{\ \ c0} Riem(g)_{c0}^{\ \ a0}
		\label{E:SIMPLECURVATUREDECOMPOSITION} 
		\\
& \ \
	- 4 g_{cc'}g_{dd'}g^{bb'}Riem(g)_{0b}^{\ \ cd} Riem(g)_{0b'}^{\ \ c'd'}.
	\notag
\end{align}
Next, we let 
$\underline{g}$ denote the first fundamental form of $\Sigma_t$ relative to $g$,
that is, 
$
\underline{g}_{ij} 
= 
g_{ij}
= (1 + \Psi)^{-P} \delta_{ij}
$
for $i,j=1,2,3$, where $\delta_{ij}$ denotes the standard Kronecker delta. 
We also let (recalling that $P=1$ in the present context)
\begin{align} \label{E:SECFUN}
		k_{\ j}^i
		& := - (\underline{g}^{-1})^{ia} 
			\left(
				\frac{1}{2} \mathcal{L}_{\partial_t} \underline{g}_{aj}
			\right)
		= \frac{1}{2} 
			\left\lbrace
				\partial_t \ln (1 + \Psi) 
			\right\rbrace	
			\delta_{\ j}^i
\end{align}
denote the (type $\binom{1}{1}$)
second fundamental form of $\Sigma_t$ relative to $g$,
where $\mathcal{L}_{\partial_t}$ denotes Lie differentiation
with respect to the vector field $\partial_t$
and $\delta_{\ j}^i$ denotes the standard Kronecker delta.
Standard calculations based on the Gauss and Codazzi equations 
for the Lorentzian manifold $(\mathbb{R}^{1+3},g)$
yield, see for example, \cite{iRjS2014b}*{Appendix A},
the identities
\begin{align}
		Riem(g)_{ab}^{\ \ cd} 
		& =
			k_{\ a}^c k_{\ b}^d
			- k_{\ a}^d k_{\ b}^c
			+ \triangle_{ab}^{\ \ cd},
			\label{E:RFOURALLSPATIALDECOMP} \\
		Riem(g)_{a0}^{\ \ c 0} 
		& = 
			(\partial_t \ln(1 + \Psi)) k_{\ a}^c
			+ k_{\ e}^c k_{\ a}^e
			+ \triangle_{a0}^{\ \ c0},
			\label{E:RFOURTWO0DECOMP} \\
		Riem(g)_{0b}^{\ \ cd} & = \triangle_{0b}^{\ \ cd},
		\label{E:RFOURONE0DECOMP}
	\end{align}
where the error terms are
\begin{align}
		\triangle_{ab}^{\ \ cd} & := Riem(\underline{g})_{ab}^{\ \ cd},
		 \label{E:RFOURFOURERROR} \\
		\triangle_{a0}^{\ \ c0} & 
			:= - \frac{1}{1 + \Psi} \partial_t ((1 + \Psi) k_{\ a}^c),
			\label{E:RFOURFOURTWOERROR} \\
		\triangle_{0b}^{\ \ cd} & := 
					(\underline{g}^{-1})^{ce} \partial_e (k_{\ b}^d)
				- (\underline{g}^{-1})^{de} \partial_e (k_{\ b}^c)
				\label{E:RFOURTHREEERROR} 
				\\
		& \ \ + (\underline{g}^{-1})^{ce} \Gamma_{e \ f}^{\ d} k_{\ b}^f
			- (\underline{g}^{-1})^{ce} \Gamma_{e \ b}^{\ f} k_{\ f}^d
			- (\underline{g}^{-1})^{de} \Gamma_{e \ f}^{\ c} k_{\ b}^f
			+ (\underline{g}^{-1})^{de} \Gamma_{e \ b}^{\ f} k_{\ f}^c,
			\notag
	\end{align}
	and
	\begin{align} \label{E:CHRISTOFFELSYMBOLSFIRSTFUND}
		\Gamma_{j \ k}^{\ i}
		& := 
		\frac{1}{2}
		(\underline{g}^{-1})^{ai}
		\left\lbrace
			\partial_j \underline{g}_{ak}
			+
			\partial_k \underline{g}_{ja}
			-
			\partial_a \underline{g}_{jk}
		\right\rbrace 
	\end{align}
are the Christoffel symbols\footnote{Our index conventions for the Christoffel symbols
are different than the ones used in many works on differential geometry.} 
of $\underline{g}$.
In \eqref{E:RFOURFOURERROR}, $Riem(\underline{g})$ denotes the Riemann curvature tensor of $\underline{g}$.
We note that in deriving \eqref{E:RFOURTWO0DECOMP}
and \eqref{E:RFOURFOURTWOERROR}, we have used the simple identity
\[
- \partial_t (k_{\ a}^c)
=
(\partial_t \ln(1 + \Psi)) k_{\ a}^c
- \frac{1}{1 + \Psi} \partial_t ((1 + \Psi) k_{\ a}^c).
\]

We will use the estimates of Theorem~\ref{T:STABLEFINITETIMEBREAKDOWN} 
to show that
\begin{align}
		\triangle_{ab}^{\ \ cd} 
		& := \mathcal{O}(\mathring{\upepsilon}) \frac{1}{1 + \Psi},
		 \label{E:ESTIMATERFOURFOURERROR} 
		\\
		\triangle_{a0}^{\ \ c0} & 
			:= \mathcal{O}(\mathring{\upepsilon}),
			\label{E:ESTIMATERFOURFOURTWOERROR} 
			\\
		\triangle_{0b}^{\ \ cd} 
		& := \mathcal{O}(\mathring{\upepsilon}) \frac{1}{1 + \Psi}.
				\label{E:ESTIMATERFOURTHREEERROR}
	\end{align}
The desired bound \eqref{E:BLOWUPOFKRETSCHMANN} then follows from 
\eqref{E:SIMPLECURVATUREDECOMPOSITION}, \eqref{E:SECFUN}, \eqref{E:RFOURALLSPATIALDECOMP}-\eqref{E:RFOURONE0DECOMP},
\eqref{E:ESTIMATERFOURFOURERROR}-\eqref{E:ESTIMATERFOURTHREEERROR},
the simple estimates $g^{ij} = \mathcal{O}(1) (1 + \Psi)$ and $g_{ij} = \mathcal{O}(1) (1 + \Psi)^{-1}$,
and straightforward calculations.

It remains for us to prove
\eqref{E:ESTIMATERFOURFOURERROR}-\eqref{E:ESTIMATERFOURTHREEERROR}.
To prove \eqref{E:ESTIMATERFOURFOURERROR}, 
we first use \eqref{E:CHRISTOFFELSYMBOLSFIRSTFUND}
and \eqref{E:PSIITSELFBOOTSTRAPIMPROVED}-\eqref{E:BOOTSTRAPIMPROVED}
to deduce
\begin{align} \label{E:CHIRSTSYMBOLEST}
	\Gamma_{j \ k}^{\ i}
	& = \mathcal{O}(\mathring{\upepsilon}) \frac{1}{1 + \Psi},
	& & 
	\partial_l \Gamma_{j \ k}^{\ i}
	= \mathcal{O}(\mathring{\upepsilon}) \frac{1}{(1 + \Psi)^2}.
\end{align}
Since  $Riem(\underline{g})_{ab}^{\ \ cd}$
has the schematic structure
$Riem(\underline{g})_{ab}^{\ \ cd}
= \underline{g}^{-1} \underline{\partial} \Gamma
	+ \underline{g}^{-1} \Gamma \cdot \Gamma
$
(where $\underline{\partial}$ denotes the gradient with respect to the spatial coordinates),
we deduce from \eqref{E:CHIRSTSYMBOLEST} 
and the simple estimate 
$(\underline{g}^{-1})^{ij} = \mathcal{O}(1) (1 + \Psi)$
that
\[
Riem(\underline{g})_{ab}^{\ \ cd}
= \mathcal{O}(\mathring{\upepsilon}) \frac{1}{1 + \Psi},
\]
which yields \eqref{E:ESTIMATERFOURFOURERROR}.
To prove \eqref{E:ESTIMATERFOURFOURTWOERROR},
we first use equation 
\eqref{E:SECFUN}, equation \eqref{E:MODELWAVE},
and the estimates \eqref{E:PSIITSELFBOOTSTRAPIMPROVED} and \eqref{E:BOOTSTRAPIMPROVED}
to deduce 
$
\displaystyle
\partial_t ((1 + \Psi) k_{\ a}^c)
= \frac{1}{2} \partial_t^2 \Psi \delta_{\ a}^c
= \frac{1}{2} (1 + \Psi) \Delta \Psi \delta_{\ a}^c
= \mathcal{O}(\mathring{\upepsilon}) (1 + \Psi)
$.
From this bound and \eqref{E:RFOURFOURTWOERROR},
we conclude \eqref{E:ESTIMATERFOURFOURTWOERROR}.
To prove \eqref{E:ESTIMATERFOURTHREEERROR},
we first use equation \eqref{E:SECFUN}
and the estimates \eqref{E:PSIITSELFBOOTSTRAPIMPROVED}-\eqref{E:BOOTSTRAPIMPROVED}
to deduce
\begin{align} \label{E:SEDONDFUNDEST}
	k_{\ j}^i
	& = \mathcal{O}(1) \frac{1}{1 + \Psi},
	&& 
	\partial_l k_{\ j}^i
	= \mathcal{O}(\mathring{\upepsilon}) \frac{1}{(1 + \Psi)^2}.
\end{align}
From 
\eqref{E:CHIRSTSYMBOLEST},
\eqref{E:SEDONDFUNDEST},
and the simple estimate 
$(\underline{g}^{-1})^{ij} = \mathcal{O}(1) (1 + \Psi)$, we conclude \eqref{E:ESTIMATERFOURTHREEERROR}.

\end{proof}

\section*{Acknowledgments}
Speck would like to thank Yuusuke Sugiyama 
for bringing the references  
\cites{kKyS2013,yS2013,sY2016a,sY2016b}
to his attention,
to Michael Dreher for pointing out
the references \cites{rM1996,rM1999},
to Willie Wong 
for pointing out the relevance of his work \cite{wW2016},
and to an anonymous referee for pointing out the work \cite{nLtNbT2015}.

\bibliographystyle{amsalpha}
\bibliography{JBib}

\def\cprime{$'$} \def\cprime{$'$} \def\cprime{$'$}
\begin{bibdiv}
\begin{biblist}

\bib{sA1978}{article}{
      author={Alinhac, S.},
       title={Branching of singularities for a class of hyperbolic operators},
        date={1978},
        ISSN={0022-2518},
     journal={Indiana Univ. Math. J.},
      volume={27},
      number={6},
       pages={1027\ndash 1037},
         url={http://dx.doi.org/10.1512/iumj.1978.27.27071},
      review={\MR{511257}},
}

\bib{sA1999b}{article}{
      author={Alinhac, Serge},
       title={Blowup of small data solutions for a class of quasilinear wave
  equations in two space dimensions. {II}},
        date={1999},
        ISSN={0001-5962},
     journal={Acta Math.},
      volume={182},
      number={1},
       pages={1\ndash 23},
         url={http://dx.doi.org/10.1007/BF02392822},
      review={\MR{1687180 (2000d:35148)}},
}

\bib{sA1999a}{article}{
      author={Alinhac, Serge},
       title={Blowup of small data solutions for a quasilinear wave equation in
  two space dimensions},
        date={1999},
        ISSN={0003-486X},
     journal={Ann. of Math. (2)},
      volume={149},
      number={1},
       pages={97\ndash 127},
         url={http://dx.doi.org/10.2307/121020},
      review={\MR{1680539 (2000d:35147)}},
}

\bib{sA2001b}{article}{
      author={Alinhac, Serge},
       title={The null condition for quasilinear wave equations in two space
  dimensions. {II}},
        date={2001},
        ISSN={0002-9327},
     journal={Amer. J. Math.},
      volume={123},
      number={6},
       pages={1071\ndash 1101},
  url={http://muse.jhu.edu/journals/american_journal_of_mathematics/v123/123.6alinhac.pdf},
      review={\MR{1867312 (2003e:35193)}},
}

\bib{sA2002}{incollection}{
      author={Alinhac, Serge},
       title={A minicourse on global existence and blowup of classical
  solutions to multidimensional quasilinear wave equations},
        date={2002},
   booktitle={Journ\'ees ``\'{E}quations aux {D}\'eriv\'ees {P}artielles''
  ({F}orges-les-{E}aux, 2002)},
   publisher={Univ. Nantes},
     address={Nantes},
       pages={Exp. No. I, 33},
      review={\MR{1968197 (2004b:35227)}},
}

\bib{sA2003}{article}{
      author={Alinhac, Serge},
       title={An example of blowup at infinity for a quasilinear wave
  equation},
        date={2003},
        ISSN={0303-1179},
     journal={Ast\'erisque},
      number={284},
       pages={1\ndash 91},
        note={Autour de l'analyse microlocale},
      review={\MR{2003417 (2005a:35197)}},
}

\bib{kAgN1981}{article}{
      author={Amano, Kazuo},
      author={Nakamura, Gen},
       title={Branching of singularities for degenerate hyperbolic operators
  and {S}tokes phenomena. {II}},
        date={1981},
        ISSN={0386-2194},
     journal={Proc. Japan Acad. Ser. A Math. Sci.},
      volume={57},
      number={3},
       pages={164\ndash 167},
         url={http://projecteuclid.org/euclid.pja/1195516491},
      review={\MR{618083}},
}

\bib{kAgN1982}{article}{
      author={Amano, Kazuo},
      author={Nakamura, Gen},
       title={Branching of singularities for degenerate hyperbolic operator and
  {S}tokes phenomena. {III}},
        date={1982},
        ISSN={0386-2194},
     journal={Proc. Japan Acad. Ser. A Math. Sci.},
      volume={58},
      number={10},
       pages={432\ndash 435},
         url={http://projecteuclid.org/euclid.pja/1195515797},
      review={\MR{694953}},
}

\bib{kAgN1983}{article}{
      author={Amano, Kazuo},
      author={Nakamura, Gen},
       title={Branching of singularities for degenerate hyperbolic operator and
  {S}tokes phenomena. {IV}},
        date={1983},
        ISSN={0386-2194},
     journal={Proc. Japan Acad. Ser. A Math. Sci.},
      volume={59},
      number={2},
       pages={47\ndash 50},
         url={http://projecteuclid.org/euclid.pja/1195515726},
      review={\MR{696748}},
}

\bib{kAgN1984}{article}{
      author={Amano, Kazuo},
      author={Nakamura, Gen},
       title={Branching of singularities for degenerate hyperbolic operators},
        date={1984},
        ISSN={0034-5318},
     journal={Publ. Res. Inst. Math. Sci.},
      volume={20},
      number={2},
       pages={225\ndash 275},
         url={http://dx.doi.org/10.2977/prims/1195181607},
      review={\MR{743379}},
}

\bib{aA2006}{article}{
      author={Ascanelli, Alessia},
       title={A degenerate hyperbolic equation under {L}evi conditions},
        date={2006},
        ISSN={0430-3202},
     journal={Ann. Univ. Ferrara Sez. VII Sci. Mat.},
      volume={52},
      number={2},
       pages={199\ndash 209},
         url={http://dx.doi.org/10.1007/s11565-006-0016-3},
      review={\MR{2273094}},
}

\bib{aA2007}{article}{
      author={Ascanelli, Alessia},
       title={Well posedness under {L}evi conditions for a degenerate second
  order {C}auchy problem},
        date={2007},
        ISSN={0041-8994},
     journal={Rend. Semin. Mat. Univ. Padova},
      volume={117},
       pages={113\ndash 126},
      review={\MR{2351788}},
}

\bib{cBrM2000}{article}{
      author={Boiti, Chiara},
      author={Manfrin, Renato},
       title={Some results of blow-up and formation of singularities for
  non-strictly hyperbolic equations},
        date={2000},
        ISSN={0532-8721},
     journal={Funkcial. Ekvac.},
      volume={43},
      number={1},
       pages={87\ndash 119},
         url={http://www.math.kobe-u.ac.jp/~fe/xml/mr1774370.xml},
      review={\MR{1774370}},
}

\bib{dC2007}{book}{
      author={Christodoulou, Demetrios},
       title={The formation of shocks in 3-dimensional fluids},
      series={EMS Monographs in Mathematics},
   publisher={European Mathematical Society (EMS), Z\"urich},
        date={2007},
        ISBN={978-3-03719-031-9},
         url={http://dx.doi.org/10.4171/031},
      review={\MR{2284927 (2008e:76104)}},
}

\bib{dCaL2016}{article}{
      author={Christodoulou, Demetrios},
      author={Lisibach, Andr{\'e}},
       title={Shock development in spherical symmetry},
        date={2016},
        ISSN={2199-2576},
     journal={Annals of PDE},
      volume={2},
      number={1},
       pages={1\ndash 246},
         url={http://dx.doi.org/10.1007/s40818-016-0009-1},
}

\bib{dCsM2014}{book}{
      author={Christodoulou, Demetrios},
      author={Miao, Shuang},
       title={Compressible flow and {E}uler's equations},
      series={Surveys of Modern Mathematics},
   publisher={International Press, Somerville, MA; Higher Education Press,
  Beijing},
        date={2014},
      volume={9},
        ISBN={978-1-57146-297-8},
      review={\MR{3288725}},
}

\bib{fCsS1982}{article}{
      author={Colombini, Ferruccio},
      author={Spagnolo, Sergio},
       title={An example of a weakly hyperbolic {C}auchy problem not well posed
  in {$C^{\infty }$}},
        date={1982},
        ISSN={0001-5962},
     journal={Acta Math.},
      volume={148},
       pages={243\ndash 253},
         url={http://dx.doi.org/10.1007/BF02392730},
      review={\MR{666112}},
}

\bib{oCmHwS2012}{article}{
      author={Costin, Ovidiu},
      author={Huang, Min},
      author={Schlag, Wilhelm},
       title={On the spectral properties of {$L_\pm$} in three dimensions},
        date={2012},
        ISSN={0951-7715},
     journal={Nonlinearity},
      volume={25},
      number={1},
       pages={125\ndash 164},
         url={http://dx.doi.org/10.1088/0951-7715/25/1/125},
      review={\MR{2864380}},
}

\bib{dChLsS2010}{article}{
      author={Coutand, Daniel},
      author={Lindblad, Hans},
      author={Shkoller, Steve},
       title={A priori estimates for the free-boundary 3{D} compressible
  {E}uler equations in physical vacuum},
        date={2010},
        ISSN={0010-3616},
     journal={Comm. Math. Phys.},
      volume={296},
      number={2},
       pages={559\ndash 587},
         url={http://dx.doi.org/10.1007/s00220-010-1028-5},
      review={\MR{2608125 (2011c:35629)}},
}

\bib{dCsS2011}{article}{
      author={Coutand, Daniel},
      author={Shkoller, Steve},
       title={Well-posedness in smooth function spaces for moving-boundary
  1-{D} compressible {E}uler equations in physical vacuum},
        date={2011},
        ISSN={0010-3640},
     journal={Comm. Pure Appl. Math.},
      volume={64},
      number={3},
       pages={328\ndash 366},
         url={http://dx.doi.org/10.1002/cpa.20344},
      review={\MR{2779087 (2012d:76103)}},
}

\bib{dCsS2012}{article}{
      author={Coutand, Daniel},
      author={Shkoller, Steve},
       title={Well-posedness in smooth function spaces for the moving-boundary
  three-dimensional compressible {E}uler equations in physical vacuum},
        date={2012},
        ISSN={0003-9527},
     journal={Arch. Ration. Mech. Anal.},
      volume={206},
      number={2},
       pages={515\ndash 616},
         url={http://dx.doi.org/10.1007/s00205-012-0536-1},
      review={\MR{2980528}},
}

\bib{pdAsS1992}{article}{
      author={D'Ancona, P.},
      author={Spagnolo, S.},
       title={Global solvability for the degenerate {K}irchhoff equation with
  real analytic data},
        date={1992},
        ISSN={1432-1297},
     journal={Inventiones mathematicae},
      volume={108},
      number={1},
       pages={247\ndash 262},
         url={http://dx.doi.org/10.1007/BF02100605},
}

\bib{pD1994}{article}{
      author={D'Ancona, Piero},
       title={Well posedness in {$C^\infty$} for a weakly hyperbolic second
  order equation},
        date={1994},
        ISSN={0041-8994},
     journal={Rend. Sem. Mat. Univ. Padova},
      volume={91},
       pages={65\ndash 83},
         url={http://www.numdam.org/item?id=RSMUP_1994__91__65_0},
      review={\MR{1289632}},
}

\bib{pDpT2001}{article}{
      author={D'Ancona, Piero},
      author={Trebeschi, Paola},
       title={On the local solvability for a nonlinear weakly hyperbolic
  equation with analytic coefficients},
        date={2001},
        ISSN={0360-5302},
     journal={Comm. Partial Differential Equations},
      volume={26},
      number={5-6},
       pages={779\ndash 811},
         url={http://dx.doi.org/10.1081/PDE-100002378},
      review={\MR{1843284}},
}

\bib{bDiWhY2015b}{article}{
      author={Ding, Bingbing},
      author={Witt, Ingo},
      author={Yin, Huicheng},
       title={Blowup of classical solutions for a class of 3-{D} quasilinear
  wave equations with small initial data},
        date={2015},
        ISSN={0893-4983},
     journal={Differential Integral Equations},
      volume={28},
      number={9-10},
       pages={941\ndash 970},
         url={http://projecteuclid.org/euclid.die/1435064545},
      review={\MR{3360725}},
}

\bib{bDiWhY2015}{article}{
      author={Ding, Bingbing},
      author={Witt, Ingo},
      author={Yin, Huicheng},
       title={On the lifespan and the blowup mechanism of smooth solutions to a
  class of 2-{D} nonlinear wave equations with small initial data},
        date={2015},
        ISSN={0033-569X},
     journal={Quart. Appl. Math.},
      volume={73},
      number={4},
       pages={773\ndash 796},
         url={http://dx.doi.org/10.1090/qam/1410},
      review={\MR{3432283}},
}

\bib{bDiWhY2017}{article}{
      author={Ding, Bingbing},
      author={Witt, Ingo},
      author={Yin, Huicheng},
       title={Blowup time and blowup mechanism of small data solutions to
  general 2-{D} quasilinear wave equations},
        date={2017},
        ISSN={1534-0392},
     journal={Commun. Pure Appl. Anal.},
      volume={16},
      number={3},
       pages={719\ndash 744},
         url={http://dx.doi.org/10.3934/cpaa.2017035},
      review={\MR{3623547}},
}

\bib{rD2010}{article}{
      author={Donninger, Roland},
       title={Nonlinear stability of self-similar solutions for semilinear wave
  equations},
        date={2010},
        ISSN={0360-5302},
     journal={Comm. Partial Differential Equations},
      volume={35},
      number={4},
       pages={669\ndash 684},
         url={http://dx.doi.org/10.1080/03605300903575857},
      review={\MR{2753616}},
}

\bib{rDmHjKwS2014}{article}{
      author={Donninger, Roland},
      author={Huang, Min},
      author={Krieger, Joachim},
      author={Schlag, Wilhelm},
       title={Exotic blowup solutions for the {$u^5$} focusing wave equation in
  {$\mathbb{R}^3$}},
        date={2014},
        ISSN={0026-2285},
     journal={Michigan Math. J.},
      volume={63},
      number={3},
       pages={451\ndash 501},
         url={http://dx.doi.org/10.1307/mmj/1409932630},
      review={\MR{3255688}},
}

\bib{rDjK2013}{article}{
      author={Donninger, Roland},
      author={Krieger, Joachim},
       title={Nonscattering solutions and blowup at infinity for the critical
  wave equation},
        date={2013},
        ISSN={0025-5831},
     journal={Math. Ann.},
      volume={357},
      number={1},
       pages={89\ndash 163},
         url={http://dx.doi.org/10.1007/s00208-013-0898-1},
      review={\MR{3084344}},
}

\bib{rDbS2012}{article}{
      author={Donninger, Roland},
      author={Sch{\"o}rkhuber, Birgit},
       title={Stable self-similar blow up for energy subcritical wave
  equations},
        date={2012},
        ISSN={1548-159X},
     journal={Dyn. Partial Differ. Equ.},
      volume={9},
      number={1},
       pages={63\ndash 87},
         url={http://dx.doi.org/10.4310/DPDE.2012.v9.n1.a3},
      review={\MR{2909934}},
}

\bib{rDbS2014}{article}{
      author={Donninger, Roland},
      author={Sch{\"o}rkhuber, Birgit},
       title={Stable blow up dynamics for energy supercritical wave equations},
        date={2014},
        ISSN={0002-9947},
     journal={Trans. Amer. Math. Soc.},
      volume={366},
      number={4},
       pages={2167\ndash 2189},
         url={http://dx.doi.org/10.1090/S0002-9947-2013-06038-2},
      review={\MR{3152726}},
}

\bib{mD1999}{thesis}{
      author={Dreher, Michael},
       title={Local solutions to quasilinear weakly hyperbolic differential
  equations},
        type={Ph.D. Thesis},
        date={1999},
}

\bib{tDcKfM2012}{article}{
      author={Duyckaerts, Thomas},
      author={Kenig, Carlos},
      author={Merle, Frank},
       title={Universality of the blow-up profile for small type {II} blow-up
  solutions of the energy-critical wave equation: the nonradial case},
        date={2012},
        ISSN={1435-9855},
     journal={J. Eur. Math. Soc. (JEMS)},
      volume={14},
      number={5},
       pages={1389\ndash 1454},
         url={http://dx.doi.org/10.4171/JEMS/336},
      review={\MR{2966655}},
}

\bib{yElMmM1986}{article}{
      author={Ebihara, Y.},
      author={Medeiros, L.~A.},
      author={Miranda, M.~Milla},
       title={Local solutions for a nonlinear degenerate hyperbolic equation},
        date={1986},
        ISSN={0362-546X},
     journal={Nonlinear Anal.},
      volume={10},
      number={1},
       pages={27\ndash 40},
         url={http://dx.doi.org/10.1016/0362-546X(86)90009-X},
      review={\MR{820656}},
}

\bib{yE1985}{article}{
      author={Ebihara, Yukiyoshi},
       title={On the existence of local smooth solutions for some degenerate
  quasilinear hyperbolic equations},
        date={1985},
        ISSN={0001-3765},
     journal={An. Acad. Brasil. Ci\^enc.},
      volume={57},
      number={2},
       pages={145\ndash 152},
      review={\MR{825353}},
}

\bib{rH1982}{article}{
      author={Hamilton, Richard~S.},
       title={The inverse function theorem of nash and moser},
        date={198207},
     journal={Bull. Amer. Math. Soc. (N.S.)},
      volume={7},
      number={1},
       pages={65\ndash 222},
         url={http://projecteuclid.org/euclid.bams/1183549049},
}

\bib{qH2010}{article}{
      author={Han, Qing},
       title={Energy estimates for a class of degenerate hyperbolic equations},
        date={2010},
        ISSN={0025-5831},
     journal={Math. Ann.},
      volume={347},
      number={2},
       pages={339\ndash 364},
         url={http://dx.doi.org/10.1007/s00208-009-0437-2},
      review={\MR{2606940}},
}

\bib{qHjxHcsL2003}{article}{
      author={Han, Qing},
      author={Hong, Jia-Xing},
      author={Lin, Chang-Shou},
       title={Local isometric embedding of surfaces with nonpositive gaussian
  curvature},
        date={200301},
     journal={J. Differential Geom.},
      volume={63},
      number={3},
       pages={475\ndash 520},
         url={http://projecteuclid.org/euclid.jdg/1090426772},
}

\bib{qHjHcL2006}{article}{
      author={Han, Qing},
      author={Hong, Jia-Xing},
      author={Lin, Chang-Shou},
       title={On the {C}auchy problem of degenerate hyperbolic equations},
        date={2006},
        ISSN={0002-9947},
     journal={Trans. Amer. Math. Soc.},
      volume={358},
      number={9},
       pages={4021\ndash 4044},
         url={http://dx.doi.org/10.1090/S0002-9947-05-03791-8},
      review={\MR{2219008}},
}

\bib{qHyL2015}{article}{
      author={Han, Qing},
      author={Liu, Yannan},
       title={Degenerate hyperbolic equations with lower degree degeneracy},
        date={2015},
        ISSN={0002-9939},
     journal={Proc. Amer. Math. Soc.},
      volume={143},
      number={2},
       pages={567\ndash 580},
         url={http://dx.doi.org/10.1090/S0002-9939-2014-12288-X},
      review={\MR{3283645}},
}

\bib{tH2012}{thesis}{
      author={Herrmann, Torsten},
       title={{$H^{\infty}$} well-posedness for degenerate $p$-evolution
  operators},
        type={Ph.D. Thesis},
        date={2012},
}

\bib{tHmRkY2013}{incollection}{
      author={Herrmann, Torsten},
      author={Reissig, Michael},
      author={Yagdjian, Karen},
       title={{$H^\infty$} well-posedness for degenerate {$p$}-evolution models
  of higher order with time-dependent coefficients},
        date={2013},
   booktitle={Progress in partial differential equations},
      series={Springer Proc. Math. Stat.},
      volume={44},
   publisher={Springer, Cham},
       pages={125\ndash 151},
         url={http://dx.doi.org/10.1007/978-3-319-00125-8_6},
      review={\MR{3097358}},
}

\bib{gHsKjSwW2016}{article}{
      author={Holzegel, Gustav},
      author={Klainerman, Sergiu},
      author={Speck, Jared},
      author={Wong, Willie Wai-Yeung},
       title={Small-data shock formation in solutions to 3d quasilinear wave
  equations: An overview},
        date={2016},
     journal={Journal of Hyperbolic Differential Equations},
      volume={13},
      number={01},
       pages={1\ndash 105},
  eprint={http://www.worldscientific.com/doi/pdf/10.1142/S0219891616500016},
  url={http://www.worldscientific.com/doi/abs/10.1142/S0219891616500016},
}

\bib{lH1997}{book}{
      author={H{\"o}rmander, Lars},
       title={Lectures on nonlinear hyperbolic differential equations},
      series={Math\'ematiques \& Applications (Berlin) [Mathematics \&
  Applications]},
   publisher={Springer-Verlag},
     address={Berlin},
        date={1997},
      volume={26},
        ISBN={3-540-62921-1},
      review={\MR{MR1466700 (98e:35103)}},
}

\bib{hIkY2002}{article}{
      author={Ishida, Haruhisa},
      author={Yagdjian, Karen},
       title={On a sharp {L}evi condition in {G}evrey classes for some
  infinitely degenerate hyperbolic equations and its necessity},
        date={2002},
        ISSN={0034-5318},
     journal={Publ. Res. Inst. Math. Sci.},
      volume={38},
      number={2},
       pages={265\ndash 287},
         url={http://projecteuclid.org/euclid.prims/1145476338},
      review={\MR{1903740}},
}

\bib{vI1975}{article}{
      author={Ivri{\u\i}, V.~Ja.},
       title={Conditions for correctness in {G}evrey classes of the {C}auchy
  problem for hyperbolic operators with characteristics of variable
  multiplicity},
        date={1975},
        ISSN={0002-3264},
     journal={Dokl. Akad. Nauk SSSR},
      volume={221},
      number={6},
       pages={1253\ndash 1255},
      review={\MR{0397180}},
}

\bib{jJnM2009}{article}{
      author={Jang, Juhi},
      author={Masmoudi, Nader},
       title={Well-posedness for compressible {E}uler equations with physical
  vacuum singularity},
        date={2009},
        ISSN={0010-3640},
     journal={Comm. Pure Appl. Math.},
      volume={62},
      number={10},
       pages={1327\ndash 1385},
         url={http://dx.doi.org/10.1002/cpa.20285},
      review={\MR{2547977 (2010j:35384)}},
}

\bib{jJnM2011}{incollection}{
      author={Jang, Juhi},
      author={Masmoudi, Nader},
       title={Vacuum in gas and fluid dynamics},
        date={2011},
   booktitle={Nonlinear conservation laws and applications},
      series={IMA Vol. Math. Appl.},
      volume={153},
   publisher={Springer, New York},
       pages={315\ndash 329},
         url={http://dx.doi.org/10.1007/978-1-4419-9554-4_17},
      review={\MR{2857004 (2012i:35230)}},
}

\bib{fJ1974}{article}{
      author={John, Fritz},
       title={Formation of singularities in one-dimensional nonlinear wave
  propagation},
        date={1974},
        ISSN={0010-3640},
     journal={Comm. Pure Appl. Math.},
      volume={27},
       pages={377\ndash 405},
      review={\MR{0369934 (51 \#6163)}},
}

\bib{fJ1981}{article}{
      author={John, Fritz},
       title={Blow-up for quasilinear wave equations in three space
  dimensions},
        date={1981},
        ISSN={0010-3640},
     journal={Comm. Pure Appl. Math.},
      volume={34},
      number={1},
       pages={29\ndash 51},
         url={http://dx.doi.org/10.1002/cpa.3160340103},
      review={\MR{600571 (83d:35096)}},
}

\bib{fJ1984}{incollection}{
      author={John, Fritz},
       title={Formation of singularities in elastic waves},
        date={1984},
   booktitle={Trends and applications of pure mathematics to mechanics
  ({P}alaiseau, 1983)},
      series={Lecture Notes in Phys.},
      volume={195},
   publisher={Springer},
     address={Berlin},
       pages={194\ndash 210},
         url={http://dx.doi.org/10.1007/3-540-12916-2_58},
      review={\MR{755727 (85h:73018)}},
}

\bib{kKyS2013}{article}{
      author={Kato, Keiichi},
      author={Sugiyama, Yuusuke},
       title={Blow up of solutions to the second sound equation in one space
  dimension},
        date={2013},
     journal={Kyushu Journal of Mathematics},
      volume={67},
      number={1},
       pages={129\ndash 142},
}

\bib{ceKfM2008}{article}{
      author={Kenig, Carlos~E.},
      author={Merle, Frank},
       title={Global well-posedness, scattering and blow-up for the
  energy-critical focusing non-linear wave equation},
        date={2008},
        ISSN={1871-2509},
     journal={Acta Mathematica},
      volume={201},
      number={2},
       pages={147\ndash 212},
         url={http://dx.doi.org/10.1007/s11511-008-0031-6},
}

\bib{rKbSmV2014}{article}{
      author={Killip, Rowan},
      author={Stovall, Betsy},
      author={Visan, Monica},
       title={Blowup behaviour for the nonlinear {K}lein-{G}ordon equation},
        date={2014},
        ISSN={0025-5831},
     journal={Math. Ann.},
      volume={358},
      number={1-2},
       pages={289\ndash 350},
         url={http://dx.doi.org/10.1007/s00208-013-0960-z},
      review={\MR{3157999}},
}

\bib{sKaM1980}{article}{
      author={Klainerman, Sergiu},
      author={Majda, Andrew},
       title={Formation of singularities for wave equations including the
  nonlinear vibrating string},
        date={1980},
        ISSN={0010-3640},
     journal={Comm. Pure Appl. Math.},
      volume={33},
      number={3},
       pages={241\ndash 263},
         url={http://dx.doi.org/10.1002/cpa.3160330304},
      review={\MR{562736 (81f:35080)}},
}

\bib{jKwSdD2008}{article}{
      author={Krieger, J.},
      author={Schlag, W.},
      author={Tataru, D.},
       title={Renormalization and blow up for charge one equivariant critical
  wave maps},
        date={2008},
        ISSN={0020-9910},
     journal={Invent. Math.},
      volume={171},
      number={3},
       pages={543\ndash 615},
         url={http://dx.doi.org/10.1007/s00222-007-0089-3},
      review={\MR{2372807}},
}

\bib{jKkNwS2013}{article}{
      author={Krieger, Joachim},
      author={Nakanishi, Kenji},
      author={Schlag, Wilhelm},
       title={Global dynamics away from the ground state for the
  energy-critical nonlinear wave equation},
        date={2013},
        ISSN={0002-9327},
     journal={Amer. J. Math.},
      volume={135},
      number={4},
       pages={935\ndash 965},
         url={http://dx.doi.org/10.1353/ajm.2013.0034},
      review={\MR{3086065}},
}

\bib{jKkNwS2013b}{article}{
      author={Krieger, Joachim},
      author={Nakanishi, Kenji},
      author={Schlag, Wilhelm},
       title={Global dynamics of the nonradial energy-critical wave equation
  above the ground state energy},
        date={2013},
        ISSN={1078-0947},
     journal={Discrete Contin. Dyn. Syst.},
      volume={33},
      number={6},
       pages={2423\ndash 2450},
      review={\MR{3007693}},
}

\bib{jKkNwS2014}{article}{
      author={Krieger, Joachim},
      author={Nakanishi, Kenji},
      author={Schlag, Wilhelm},
       title={Threshold phenomenon for the quintic wave equation in three
  dimensions},
        date={2014},
        ISSN={0010-3616},
     journal={Comm. Math. Phys.},
      volume={327},
      number={1},
       pages={309\ndash 332},
         url={http://dx.doi.org/10.1007/s00220-014-1900-9},
      review={\MR{3177940}},
}

\bib{jKkNwS2015}{article}{
      author={Krieger, Joachim},
      author={Nakanishi, Kenji},
      author={Schlag, Wilhelm},
       title={Center-stable manifold of the ground state in the energy space
  for the critical wave equation},
        date={2015},
        ISSN={0025-5831},
     journal={Math. Ann.},
      volume={361},
      number={1-2},
       pages={1\ndash 50},
         url={http://dx.doi.org/10.1007/s00208-014-1059-x},
      review={\MR{3302610}},
}

\bib{jKwS2014}{article}{
      author={Krieger, Joachim},
      author={Schlag, Wilhelm},
       title={Full range of blow up exponents for the quintic wave equation in
  three dimensions},
        date={2014},
        ISSN={0021-7824},
     journal={J. Math. Pures Appl. (9)},
      volume={101},
      number={6},
       pages={873\ndash 900},
         url={http://dx.doi.org/10.1016/j.matpur.2013.10.008},
      review={\MR{3205646}},
}

\bib{jKwSdT2009}{article}{
      author={Krieger, Joachim},
      author={Schlag, Wilhelm},
      author={Tataru, Daniel},
       title={Slow blow-up solutions for the {$H^1(\mathbb{R}^3)$} critical
  focusing semilinear wave equation},
        date={2009},
        ISSN={0012-7094},
     journal={Duke Math. J.},
      volume={147},
      number={1},
       pages={1\ndash 53},
         url={http://dx.doi.org/10.1215/00127094-2009-005},
      review={\MR{2494455}},
}

\bib{pL1957}{article}{
      author={Lax, P.~D.},
       title={Hyperbolic systems of conservation laws. {II}},
        date={1957},
        ISSN={0010-3640},
     journal={Comm. Pure Appl. Math.},
      volume={10},
       pages={537\ndash 566},
      review={\MR{0093653 (20 \#176)}},
}

\bib{nLtNbT2015}{article}{
      author={{Lerner}, N.},
      author={{Nguyen}, T.~T.},
      author={{Texier}, B.},
       title={{The onset of instability in first-order systems}},
        date={2015-04},
     journal={To appear in Journal of the European Mathematical Society;
  preprint available},
      eprint={https://arxiv.org/abs/1504.04477},
}

\bib{hL1974}{article}{
      author={Levine, Howard~A.},
       title={Instability and nonexistence of global solutions to nonlinear
  wave equations of the form {$Pu_{tt}=-Au+{\mathscr{F}}(u)$}},
        date={1974},
        ISSN={0002-9947},
     journal={Trans. Amer. Math. Soc.},
      volume={192},
       pages={1\ndash 21},
      review={\MR{0344697}},
}

\bib{hL1998}{article}{
      author={Lindblad, Hans},
       title={Counterexamples to local existence for quasilinear wave
  equations},
        date={1998},
        ISSN={1073-2780},
     journal={Math. Res. Lett.},
      volume={5},
      number={5},
       pages={605\ndash 622},
      review={\MR{MR1666844 (2000a:35171)}},
}

\bib{hL2008}{article}{
      author={Lindblad, Hans},
       title={Global solutions of quasilinear wave equations},
        date={2008},
        ISSN={0002-9327},
     journal={Amer. J. Math.},
      volume={130},
      number={1},
       pages={115\ndash 157},
         url={http://dx.doi.org/10.1353/ajm.2008.0009},
      review={\MR{2382144 (2009b:58062)}},
}

\bib{jL2013}{article}{
      author={{Luk}, Jonathan},
       title={{Weak null singularities in general relativity}},
        date={2013-11},
     journal={ArXiv e-prints},
      eprint={https://arxiv.org/abs/1311.4970},
}

\bib{rM1996}{article}{
      author={Manfrin, Renato},
       title={Local solvability in {$C^\infty$} for quasi-linear weakly
  hyperbolic equations of second order},
        date={1996},
        ISSN={0360-5302},
     journal={Comm. Partial Differential Equations},
      volume={21},
      number={9-10},
       pages={1487\ndash 1519},
         url={http://dx.doi.org/10.1080/03605309608821236},
      review={\MR{1410839}},
}

\bib{rM1999}{article}{
      author={Manfrin, Renato},
       title={Well posedness in the {$C^\infty$} class for {$u_{tt}=a(u)\Delta
  u$}},
        date={1999},
        ISSN={0362-546X},
     journal={Nonlinear Anal.},
      volume={36},
      number={2, Ser. A: Theory Methods},
       pages={177\ndash 212},
         url={http://dx.doi.org/10.1016/S0362-546X(97)00703-7},
      review={\MR{1668860}},
}

\bib{yMfMpRjS2014}{article}{
      author={Martel, Yvan},
      author={Merle, Frank},
      author={Rapha\"{e}l, Pierre},
      author={Szeftel, J{\'{e}}r{\'{e}}mie},
       title={Near soliton dynamics and singularity formation for {$L^2$}
  critical problems},
        date={2014},
     journal={Russian Mathematical Surveys},
      volume={69},
      number={2},
       pages={261},
         url={http://stacks.iop.org/0036-0279/69/i=2/a=261},
}

\bib{kNwS2011a}{article}{
      author={Nakanishi, K.},
      author={Schlag, W.},
       title={Global dynamics above the ground state energy for the focusing
  nonlinear {K}lein-{G}ordon equation},
        date={2011},
        ISSN={0022-0396},
     journal={J. Differential Equations},
      volume={250},
      number={5},
       pages={2299\ndash 2333},
         url={http://dx.doi.org/10.1016/j.jde.2010.10.027},
      review={\MR{2756065}},
}

\bib{kNwS2012a}{article}{
      author={Nakanishi, K.},
      author={Schlag, W.},
       title={Global dynamics above the ground state for the nonlinear
  {K}lein-{G}ordon equation without a radial assumption},
        date={2012},
        ISSN={0003-9527},
     journal={Arch. Ration. Mech. Anal.},
      volume={203},
      number={3},
       pages={809\ndash 851},
         url={http://dx.doi.org/10.1007/s00205-011-0462-7},
      review={\MR{2928134}},
}

\bib{kNwS2012b}{article}{
      author={Nakanishi, K.},
      author={Schlag, W.},
       title={Invariant manifolds around soliton manifolds for the nonlinear
  {K}lein-{G}ordon equation},
        date={2012},
        ISSN={0036-1410},
     journal={SIAM J. Math. Anal.},
      volume={44},
      number={2},
       pages={1175\ndash 1210},
         url={http://dx.doi.org/10.1137/11082720X},
      review={\MR{2914265}},
}

\bib{kNwS2011b}{book}{
      author={Nakanishi, Kenji},
      author={Schlag, Wilhelm},
       title={Invariant manifolds and dispersive {H}amiltonian evolution
  equations},
      series={Zurich Lectures in Advanced Mathematics},
   publisher={European Mathematical Society (EMS), Z\"urich},
        date={2011},
        ISBN={978-3-03719-095-1},
         url={http://dx.doi.org/10.4171/095},
      review={\MR{2847755}},
}

\bib{tN1980}{article}{
      author={Nishitani, Tatsuo},
       title={The {C}auchy problem for weakly hyperbolic equations of second
  order},
        date={1980},
        ISSN={0360-5302},
     journal={Comm. Partial Differential Equations},
      volume={5},
      number={12},
       pages={1273\ndash 1296},
         url={http://dx.doi.org/10.1080/03605308008820169},
      review={\MR{593968}},
}

\bib{tN1984}{article}{
      author={Nishitani, Tatsuo},
       title={A necessary and sufficient condition for the hyperbolicity of
  second order equations with two independent variables},
        date={1984},
        ISSN={0023-608X},
     journal={J. Math. Kyoto Univ.},
      volume={24},
      number={1},
       pages={91\ndash 104},
      review={\MR{737827}},
}

\bib{oO1957}{article}{
      author={Ole{\u{\i}}nik, O.~A.},
       title={Discontinuous solutions of non-linear differential equations},
        date={1957},
        ISSN={0042-1316},
     journal={Uspehi Mat. Nauk (N.S.)},
      volume={12},
      number={3(75)},
       pages={3\ndash 73},
      review={\MR{0094541 (20 \#1055)}},
}

\bib{oO1970}{article}{
      author={Ole{\u\i}nik, O.~A.},
       title={On the {C}auchy problem for weakly hyperbolic equations},
        date={1970},
        ISSN={0010-3640},
     journal={Comm. Pure Appl. Math.},
      volume={23},
       pages={569\ndash 586},
      review={\MR{0264227}},
}

\bib{lPdS1975}{article}{
      author={Payne, L.~E.},
      author={Sattinger, D.~H.},
       title={Saddle points and instability of nonlinear hyperbolic equations},
        date={1975},
        ISSN={0021-2172},
     journal={Israel J. Math.},
      volume={22},
      number={3-4},
       pages={273\ndash 303},
      review={\MR{0402291}},
}

\bib{pRiR2012}{article}{
      author={Rapha{\"e}l, Pierre},
      author={Rodnianski, Igor},
       title={Stable blow up dynamics for the critical co-rotational wave maps
  and equivariant {Y}ang-{M}ills problems},
        date={2012},
        ISSN={0073-8301},
     journal={Publ. Math. Inst. Hautes \'Etudes Sci.},
      volume={115},
       pages={1\ndash 122},
         url={http://dx.doi.org/10.1007/s10240-011-0037-z},
      review={\MR{2929728}},
}

\bib{bR1860}{article}{
      author={Riemann, Bernhard},
       title={{\"{U}}ber die {F}ortpflanzung ebener {L}uftwellen von endlicher
  {S}chwingungsweite},
        date={1860},
     journal={Abhandlungen der K�niglichen Gesellschaft der Wissenschaften in
  G�ttingen},
      volume={8},
       pages={43\ndash 66},
         url={http://eudml.org/doc/135717},
}

\bib{iRjS2014a}{article}{
      author={Rodnianski, Igor},
      author={Speck, Jared},
       title={{A Regime of Linear Stability for the Einstein-Scalar Field
  System with Applications to Nonlinear Big Bang Formation}},
        date={2014-07},
     journal={To appear in Annals of Mathematics; preprint available},
      eprint={https://arxiv.org/abs/1407.6293},
}

\bib{iRjS2014b}{article}{
      author={Rodnianski, Igor},
      author={Speck, Jared},
       title={{Stable Big Bang Formation in Near-FLRW Solutions to the
  Einstein-Scalar Field and Einstein-Stiff Fluid Systems}},
        date={2014-07},
     journal={ArXiv e-prints},
      eprint={https://arxiv.org/abs/1407.6298},
}

\bib{iRjS2010}{article}{
      author={Rodnianski, Igor},
      author={Sterbenz, Jacob},
       title={On the formation of singularities in the critical {${\rm O}(3)$}
  {$\sigma$}-model},
        date={2010},
        ISSN={0003-486X},
     journal={Ann. of Math. (2)},
      volume={172},
      number={1},
       pages={187\ndash 242},
         url={http://dx.doi.org/10.4007/annals.2010.172.187},
      review={\MR{2680419}},
}

\bib{zRiWhY2016}{article}{
      author={{Ruan}, Z.},
      author={{Witt}, I.},
      author={{Yin}, H.},
       title={{Minimal regularity solutions of semilinear generalized Tricomi
  equations}},
        date={2016-08},
     journal={ArXiv e-prints},
      eprint={1608.01826},
}

\bib{zRiWhY2015b}{article}{
      author={Ruan, Zhuoping},
      author={Witt, Ingo},
      author={Yin, Huicheng},
       title={On the existence and cusp singularity of solutions to semilinear
  generalized {T}ricomi equations with discontinuous initial data},
        date={2015},
        ISSN={0219-1997},
     journal={Commun. Contemp. Math.},
      volume={17},
      number={3},
       pages={1450028, 49},
         url={http://dx.doi.org/10.1142/S021919971450028X},
      review={\MR{3325043}},
}

\bib{zRiWhY2015a}{article}{
      author={Ruan, Zhuoping},
      author={Witt, Ingo},
      author={Yin, Huicheng},
       title={On the existence of low regularity solutions to semilinear
  generalized {T}ricomi equations in mixed type domains},
        date={2015},
        ISSN={0022-0396},
     journal={J. Differential Equations},
      volume={259},
      number={12},
       pages={7406\ndash 7462},
         url={http://dx.doi.org/10.1016/j.jde.2015.08.025},
      review={\MR{3401601}},
}

\bib{mRmR2016}{article}{
      author={Ruziev, Menglibay},
      author={Reissig, Michael},
       title={Tricomi type equations with terms of lower order},
        date={2016},
        ISSN={1752-3583},
     journal={Int. J. Dyn. Syst. Differ. Equ.},
      volume={6},
      number={1},
       pages={1\ndash 15},
         url={http://dx.doi.org/10.1504/IJDSDE.2016.074572},
      review={\MR{3463879}},
}

\bib{tS1985}{article}{
      author={Sideris, Thomas~C.},
       title={Formation of singularities in three-dimensional compressible
  fluids},
        date={1985},
        ISSN={0010-3616},
     journal={Comm. Math. Phys.},
      volume={101},
      number={4},
       pages={475\ndash 485},
      review={\MR{MR815196 (87d:35127)}},
}

\bib{jS2008c}{thesis}{
      author={Speck, Jared},
       title={On the questions of local and global well-posedness for the
  hyperbolic pdes occurring in some relativistic theories of gravity and
  electromagnetism},
        type={{PhD} dissertation},
     address={Piscataway, NJ},
        date={2008},
}

\bib{jSgHjLwW2016}{article}{
      author={Speck, Jared},
      author={Holzegel, Gustav},
      author={Luk, Jonathan},
      author={Wong, Willie},
       title={Stable shock formation for nearly simple outgoing plane symmetric
  waves},
        date={2016},
        ISSN={2199-2576},
     journal={Annals of PDE},
      volume={2},
      number={2},
       pages={1\ndash 198},
         url={http://dx.doi.org/10.1007/s40818-016-0014-4},
}

\bib{mS2003}{article}{
      author={Struwe, Michael},
       title={Equivariant wave maps in two space dimensions},
        date={2003},
        ISSN={1097-0312},
     journal={Communications on Pure and Applied Mathematics},
      volume={56},
      number={7},
       pages={815\ndash 823},
         url={http://dx.doi.org/10.1002/cpa.10074},
}

\bib{sY2016b}{article}{
      author={{Sugiyama}, Y.},
       title={{Degeneracy in finite time of 1D quasilinear wave equations II}},
        date={2016-01},
     journal={ArXiv e-prints},
      eprint={https://arxiv.org/abs/1601.05191},
}

\bib{yS2013}{article}{
      author={Sugiyama, Yuusuke},
       title={Global existence of solutions to some quasilinear wave equation
  in one space dimension},
        date={201305},
     journal={Differential Integral Equations},
      volume={26},
      number={5/6},
       pages={487\ndash 504},
         url={http://projecteuclid.org/euclid.die/1363266076},
}

\bib{sY2016a}{article}{
      author={Sugiyama, Yuusuke},
       title={Degeneracy in finite time of 1d quasilinear wave equations},
        date={2016},
     journal={SIAM Journal on Mathematical Analysis},
      volume={48},
      number={2},
       pages={847\ndash 860},
      eprint={http://dx.doi.org/10.1137/15M1016369},
         url={http://dx.doi.org/10.1137/15M1016369},
}

\bib{kTyT1980}{article}{
      author={Taniguchi, Kazuo},
      author={Tozaki, Yoshiharu},
       title={A hyperbolic equation with double characteristics which has a
  solution with branching singularities},
        date={1980},
        ISSN={0025-5513},
     journal={Math. Japon.},
      volume={25},
      number={3},
       pages={279\ndash 300},
      review={\MR{586523}},
}

\bib{wW2016}{article}{
      author={{Wai-Yeung Wong}, W.},
       title={{Singularities of axially symmetric time-like minimal
  submanifolds in Minkowski space}},
        date={2016-07},
     journal={ArXiv e-prints},
      eprint={https://arxiv.org/abs/1607.06846},
}

\bib{kY1989}{article}{
      author={Yagdzhyan, K.~A.},
       title={The parametrix of the {C}auchy problem for hyperbolic operators
  that are degenerate with respect to the space variables},
        date={1989},
        ISSN={0002-3043},
     journal={Izv. Akad. Nauk Armyan. SSR Ser. Mat.},
      volume={24},
      number={5},
       pages={421\ndash 432, 516},
      review={\MR{1062112}},
}

\end{biblist}
\end{bibdiv}

\end{document}